\documentclass[12pt, letterpaper]{article}

\usepackage[T1]{fontenc}
\usepackage[english]{babel}
\usepackage[utf8]{inputenc}

\usepackage[margin=1truein]{geometry}

\usepackage{amsfonts,amssymb,amsmath,amsthm}
\usepackage{mathtools}
\usepackage{mathrsfs}
\usepackage{enumerate}
\usepackage{verbatim}
\usepackage{caption}

\usepackage[pdftex,colorlinks=true,linkcolor=blue,citecolor=blue,urlcolor=red,unicode=true,hyperfootnotes=true,bookmarksnumbered]{hyperref}

\theoremstyle{plain}
\newtheorem{theorem}{Theorem}[section]
\newtheorem{lemma}[theorem]{Lemma}
\newtheorem{claim}[theorem]{Claim}
\newtheorem{conjecture}[theorem]{Conjecture}
\newtheorem{proposition}[theorem]{Proposition}
\newtheorem{corollary}[theorem]{Corollary}

\newtheorem{obs}{Observation}[section]

\theoremstyle{definition} 

\newtheorem*{remark*}{Remark}
\newtheorem{definition}{Definition}
\newtheorem*{definition*}{Definition}

\renewcommand{\emptyset}{\varnothing}

\renewcommand{\epsilon}{\varepsilon}
\renewcommand{\phi}{\varphi}
\renewcommand{\kappa}{\varkappa}
\newcommand{\R}{\mathbb{R}}
\newcommand{\etamat}{\eta_{\mathrm{mat}}}

\newcommand{\F}{\mathbb{F}}
\newcommand{\Q}{\mathbb{Q}}
\newcommand{\Z}{\mathbb{Z}}
\renewcommand{\P}{\mathbb{P}}
\newcommand{\E}{\mathbb{E}}
\newcommand{\Nz}{\Z_{\ge0}}

\newcommand{\suchthat}{\mathrel{\mkern2mu:\mkern2mu}}
\newcommand{\fbrc}[2]{\left\{#1 \suchthat #2\right\}}

\newcommand{\fbr}[1]{\left\{#1\right\}}
\newcommand{\br}[1]{\left({#1}\right)}
\newcommand{\lbr}[1]{\left|{#1}\right|}
\newcommand{\norminf}[1]{\lVert #1\rVert_{\infty}}

\newcommand{\floor}[1]{\left\lfloor{#1}\right\rfloor}
\newcommand{\ceil}[1]{\left\lceil{#1}\right\rceil}

\newcommand{\wt}[1]{\widetilde{#1}}

\newcommand{\mc}[1]{\mathcal{#1}}

\newcommand{\cl}{\operatorname{cl}}
\newcommand{\seq}{\subseteq}
\newcommand{\Flats}{\mathrm{Flats}}
\newcommand{\Ent}{\operatorname{H}}
\newcommand{\Gent}[1]{\mc{H}^{\text{ent}}_{#1}}
\newcommand{\clGent}[1]{\overline{\Gent{#1}}}

\newcommand{\rk}{\operatorname{rk}}
\newcommand{\frk}[2]{\operatorname{frk}_{#2}(#1)}
\newcommand{\supp}{\operatorname{supp}}

\title{Almost entropic matroids and $(k,L)$-systems}

\author{
    Andrey Kupavskii\footnote{Moscow Institute of Physics and Technology, Russia, St. Petersburg State University; Email: {\tt kupavskii@ya.ru}},
    Nikolai Terekhov\footnote{Laboratory of combinatorial and geometric structures, Moscow Institute of Physics and Technology, Dolgoprudny, Russia; Email: {\tt nikolayterek@gmail.com}}
}

\begin{document}

\maketitle

\begin{abstract}
    For an integer $k$ and a set of integers $L$, a \emph{$(k,L)$-system} is a family of $k$-element sets such that the size of the intersection of any two distinct members belongs to $L$. Let $m(n,k,L)$ denote the maximum size of a $(k,L)$-system on a ground set of $n$ elements.

    In this work, we establish a close connection between the asymptotic growth of $m(n,k,L)$ and almost entropic matroids. As a consequence, we obtain a general computable upper bound on the asymptotics of $m(n,k,L)$ that unifies and strengthens all previously known upper bounds.
\end{abstract}

\section{Introduction}

For an integer $n\ge 1$, let $[n]=\{1,2,\ldots,n\}$. We write $[n]^k$ for the $k$-fold Cartesian product of $[n]$ and regard its elements as vectors $F=(F_1,\ldots,F_k)$. For a set $T$ and an integer $k\ge 0$, write
\(
    \binom{T}{k}=\fbrc{S\seq T}{\lbr{S}=k}
\).

We study families of sets with restrictions on their pairwise intersections, a classical topic in extremal combinatorics, design theory, and coding theory. The central object is the following.

\begin{definition}
    Let $n\ge k\ge 1$ be integers and let $L\seq\{0,\ldots,k-1\}$. A family $\mc{F}\seq\binom{[n]}{k}$ is called a \emph{$(k,L)$-system} if, for every distinct $F,F'\in\mc{F}$, we have
    \(
        \lbr{F\cap F'}\in L
    \).
    Define
    \[
        m(n,k,L)=\max\fbrc{\lbr{\mc{F}}}{\mc{F}\seq\binom{[n]}{k}\text{ is a }(k,L)\text{-system}}.
    \]
\end{definition}
This quantity was introduced in full generality by Deza, Erd\H{o}s and Frankl~\cite{1978Deza}. Some years earlier, in 1975, Ray-Chaudhuri and Wilson~\cite{1975RayChaudhuri}, in the context of studying designs, obtained a generalization of Fisher's inequality, which can be used to get the following general bound: $m(n,k,L)\le \binom{n}{s}$ for all $n\ge k$ and $\lbr{L}=s$. Deza, Erd\H{o}s and Frankl proved a sharper bound that depended on the structure of $L$, albeit only for $n>n_0(k)$. In their paper, they introduced the highly influential Delta-system method; see the recent survey of the first author~\cite{2025Kupavskii}.

An important development for both the study of $m(n,k,L)$ and the Delta-system method came with the paper of F\"uredi~\cite{1983Furedi}. He proved a Ramsey-type result which, at a high level, states that any family of $k$-sets contains a highly homogeneous substructure that includes a constant proportion of all sets from the original family. To state the result, we need to introduce some further definitions.

We say that a family $\mc{M}\seq 2^{[k]}$ is \emph{intersection-closed} if for any $A,B\in \mc{M}$ we have $A\cap B\in \mc{M}$.

\begin{definition}
    Let $k\ge 1$ and $n\ge 1$ be integers, and let $\mc{M}\seq 2^{[k]}$ be an intersection-closed family with $[k]\in\mc{M}$. A family $\mc{F}\seq[n]^k$ is called an \emph{$\mc{M}$-system} if, for every $F,F'\in\mc{F}$, we have
    $\fbrc{i\in[k]}{F_i=F'_i}\in\mc{M}$.
    Define
    \[
        m(n,\mc{M})=
        \max\fbrc{\lbr{\mc{F}}}{\mc{F}\seq[n]^k\text{ is an }\mc{M}\text{-system}}.
    \]
\end{definition}

\begin{theorem}[F\"uredi, \cite{1983Furedi}]\label{thrm:kL-to-M}
    Let $k\ge 1$ and $L\seq\{0,\ldots,k-1\}$. Then
    \begin{equation}\label{eq:kL-to-M}
        m(n,k,L)=\Theta_{k}\!\left(
        \max\fbrc{
            m(n,\mc{M})}{
            \begin{array}{l}
                \mc{M}\seq 2^{[k]}\text{ is intersection-closed},\ [k]\in\mc{M},\\
                \lbr{A}\in L\text{ for every }A\in\mc{M}\setminus\{[k]\}
            \end{array}
        }\right).
    \end{equation}
\end{theorem}

In this theorem, it is crucial for applications that $\mc{M}$ is intersection-closed. To derive bounds on $m(n,k,L)$, we will combine the theorem with the notion of \emph{rank}. For an intersection-closed family $\mc M\seq 2^{[k]}$ with $[k]\in\mc{M}$, define $\rk(\mc{M})$ to be the size of the smallest set $C$ such that $C\not\seq A$ for any $A\in \mc{M}\setminus \fbr{[k]}$.
The following simple but very important observation relates ranks and our extremal functions.
\begin{obs}[\cite{1985Deza,1991Furedi}]\label{obs:mM-rank-bound}
    Let $\mc{M}\seq 2^{[k]}$ be an intersection-closed family with $[k]\in \mc{M}$. Then
    \(m(n,\mc{M})\le n^{\rk(\mc{M})}\).
\end{obs}
\begin{proof}
    Assume that $C\seq [k]$ is the smallest set that is not contained in any set of $\mc{M}\setminus\fbr{[k]}$. Take any $\mc{M}$-system $\mc F\seq [n]^k$. Then each member $F\in\mc F$ is uniquely determined by $\pi_C(F)$, that is, by the set of coordinates of $F$ with indices in $C$. Consequently, $\lbr{\mc F}\le n^{\lbr{C}}$.
\end{proof}

By Theorem~\ref{thrm:kL-to-M}, the rank bound naturally transfers to $(k,L)$-systems, giving the bound $m(n,k,L)=O_k(n^{\rk(k,L)})$ for $\rk(k,L)=\max\fbr{\rk(\mc{M})}$ over $\mc{M}$ as in \eqref{eq:kL-to-M}.

It is not difficult to see that $\rk(k,L)\le \lbr{L}$, thus the rank bound is asymptotically at least as strong as the bound $m(n,k,L)\le \binom{n}{\lbr{L}}$ of Ray-Chaudhuri and Wilson. Moreover, it  also achieves the order of the upper bound from the full result of Deza, Erd\H{o}s and Frankl, covering their bound in the non-divisibility case (cf.~\cite{1978Deza}).

The rank bound, however, has limitations. Frankl~\cite{1986Frankl} showed that every rational number at least $1$ occurs as the exponent of $m(n,k,L)$ for some $k$ and $L$. Thus tight upper bounds cannot, in general, be obtained solely from the integer-valued parameter $\rk(k,L)$. Previous proofs of non-integral exponents relied on case-specific arguments; see, for example,~\cite{1984Frankl,1985Furedi,1996Frankl}.

In this paper, we develop a framework that unifies all previous upper bounds and gives a general non-integral sharpening of the rank. It allows us to express the exact asymptotic exponent via almost entropic matroids.

\begin{definition}[\cite{2018Matus}]
    Let $U$ be a finite set. A matroid $N$ on $U$ is called \emph{almost entropic} if for every $\epsilon > 0$ there exists a finite-support random vector $\br{\xi_u}_{u\in U}$ such that
    \[\forall A\seq U\quad |\rk_N(A)-\Ent((\xi_a)_{a\in A})| < \epsilon,\]
    where $\Ent((\xi_a)_{a\in A})$ denotes the Shannon entropy (with logarithms to base \(2\)).
\end{definition}

A priori, it could be unclear how a rank function can be related to the entropy of random vectors. Linear matroids provide a natural example of how such a relation can arise; see Proposition~\ref{prop:linear-matroids-are-almost-entropic}.

We associate an intersection-closed family $\mc{M}\subseteq 2^{[k]}$ with a matroid $N$ by allowing the elements of $[k]$ to be ``blown up'': each $i\in[k]$ is represented by a subset of the ground set of $N$, subject to certain compatibility conditions with $\mc{M}$. The \emph{folding rank} $\frk{\mc{M}}{N}$ measures the amount of blow-up required to encode $\mc{M}$ in $N$; see Definition~\ref{def:folding-rank} for the formal definition.


The following theorem is the main structural result of the paper.

\begin{theorem}\label{thrm:main-intro}
    Let $\mc{M}\seq 2^{[k]}$ be intersection-closed with $[k]\in\mc{M}$. Then
    \[
        m(n,\mc{M})=n^{\etamat(\mc{M})+o(1)},
    \]
    where
    \begin{equation}\label{eq:etamat-def}
        \etamat(\mc{M})=
        \sup\fbrc{
            \dfrac{\rk(N)}{\frk{\mc{M}}{N}}
        }{
            N\text{ is an almost entropic matroid}
        }.
    \end{equation}
\end{theorem}

Together with Theorem~\ref{thrm:kL-to-M}, this gives the corresponding description for $(k,L)$-systems.

\begin{corollary}\label{corollary:alpha-eta-intro}
    Let $k\ge 1$ and $L\seq\{0,\ldots,k-1\}$. Then
    \[
        m(n,k,L)=n^{\etamat(k,L)+o(1)},
    \]
    where $\etamat(k,L)=\max\fbr{\etamat(\mc{M})}$ over $\mc{M}$ as in \eqref{eq:kL-to-M}.
\end{corollary}

In particular, the limits $\alpha(\mc{M})=\lim_{n\to\infty}\frac{\log m(n,\mc{M})}{\log n}$ and $\alpha(k,L)=\lim_{n\to\infty}\frac{\log m(n,k,L)}{\log n}$ always exist and are determined exactly by almost entropic matroids.

The characterization in Theorem~\ref{thrm:main-intro} is exact, but it is difficult to use directly because the class of almost entropic matroids has a complicated structure. For upper bounds, we may relax the characterization by taking the supremum in \eqref{eq:etamat-def} over all matroids, rather than only almost entropic matroids. This relaxation, denoted by $\rho(\mc{M})$, can be computed by solving a linear program; see Corollary~\ref{corollary:alpha-le-rho} and Theorem~\ref{thrm:rho-is-a-folding}.
\begin{theorem}\label{thrm:rho-upper-intro}
    Let $\mc{M}\seq 2^{[k]}$ be intersection-closed with $[k]\in\mc{M}$. Then
    \[
        \eta_{mat}(\mc M)\le \rho(\mc M),
    \]
    where $\rho(\mc M)\le \rk(\mc M)$ is the solution to a certain linear program; see Corollary~\ref{corollary:alpha-le-rho}.

    Moreover, for $k\ge 1$ and $L\seq\{0,\ldots,k-1\}$,
    \[
        \alpha(k,L)\le \rho(k,L),
    \]
    where $\rho(k,L)=\max\fbr{\rho(\mc{M})}$ over $\mc{M}$ as in \eqref{eq:kL-to-M}.
\end{theorem}

As mentioned in the statement, the quantity $\rho(k,L)$ is at most $\rk(k,L)$. Moreover, unlike the rank bound, it may be non-integral. This bound, being a maximum over finitely many linear programs, is a computable rational number for every fixed $k$ and $L$. All values of $\alpha(k,L)$ currently known to us are equal to $\rho(k,L)$. In particular, equality holds for the cases with $k\le12$ determined in~\cite{1996Frankl,2006Tokushige}, for Frankl's examples realizing arbitrary rational exponents~\cite{1986Frankl}, and for certain families with $\lbr{L}\le3$~\cite{1984Frankl,1985Furedi}. This motivates the following conjecture.

\begin{conjecture}\label{conj:akL-equals-rho}
    Let $k\ge 1$ and $L\seq\{0,\ldots,k-1\}$. Then
    \(\alpha(k,L)=\rho(k,L)\).
\end{conjecture}

Conjecture~\ref{conj:akL-equals-rho} is also consistent with the upper bounds in cases when the exact value is not known, in particular, with the general bounds in~\cite[Section~2.2]{1996Frankl} and with the bounds for $\lbr{L}\le 3$ in~\cite{1984Frankl,1985Furedi}. 

Our framework also leads to the following implications.
\begin{itemize}
\item Using folding ranks, we obtain a general lower bound
\[
    \alpha(k,L)\ge \frac{k}{k-\rk(k,L)+1},
\]
valid whenever $\rk(k,L)\ge 1$. It generalizes the bound  established by F\"uredi~\cite{1991Furedi} in the case $\rk(k,L)=2$; see Corollary~\ref{corollary:rank-lower-bound-kL}.
\item Using the upper bound $\alpha(\mc{M})\le \rho(\mc{M})$, we construct an intersection-closed family $\mc{M}$ for which
$
    \ceil{\alpha(\mc{M})}<\rk(\mc{M})-1
$.
Thus, the natural generalization of F\"uredi's conjecture~\cite{1991Furedi} from $(k,L)$-systems to $\mc{M}$-systems fails; see Corollary~\ref{corollary:counterexample-Furedi-S-all-r}.
\item Using the existence of matroids that are not almost entropic, we show that the natural $\mc{M}$-system analogue of Conjecture~\ref{conj:akL-equals-rho} fails; see Corollary~\ref{corollary:alpha-neq-rho}.
\item Finally, using the undecidability of almost entropic matroids, we show that there is no algorithm that outputs the exact value of $\alpha(\mc{M})$ for every $\mc{M}$, even assuming that $\alpha(\mc{M})$ is always rational; see Section~\ref{sec:computability}.
\end{itemize}

\paragraph{Proof method}


For a family $\mc{F}\seq [n]^k$, $A\subseteq [k]$, and $X\in\pi_A(\mc{F})$, let $\mc F[X]=\{F\in\mc{F}:\pi_A(F)=X\}$ denote the family of extensions of $X$. The non-integral upper bounds on $m(n,k,L)$ in~\cite{1984Frankl,1985Furedi,1986Frankl,1996Frankl} are obtained by controlling the sizes $|\mc F[X]|$. Depending on the case, the proofs prune those $X$ with few extensions, compare extension sizes for different choices of $A$ and $X$, or split the possible $X$ according to how many extensions they admit. The precise implementation, however, depends on the particular $k$ and $L$ under consideration.

Our approach replaces this local analysis by looking at the entropy of $|\mc{F}[X]|$ for a randomly chosen $X$, which aggregates information over all choices of $X$ at once (see the proof of Theorem~\ref{thrm:alpha-le-eta}). This entirely avoids the preliminary pruning steps needed to regularize the extension structure. The condition of being an $\mc{M}$-system translates into dependencies between projections, and entropy inequalities combine these dependencies to bound the size of the family. In particular, Shannon-type information inequalities give the upper bound in Theorem~\ref{thrm:rho-upper-intro}. Conversely, an approximation argument realizes the relevant entropy data by highly regular families (see Theorem~\ref{thrm:approximation-by-codes}), which then leads to Theorem~\ref{thrm:main-intro}.

\paragraph{Proof structure}
We first express $\alpha(\mc{M})$ in terms of almost entropic vectors compatible with $\mc{M}$  (see Definition~\ref{def:entropic-vectors} and~\ref{def:compatibility}). The entropy argument gives the upper bound (see Theorem~\ref{thrm:alpha-le-eta}), while compatible quasi-uniform approximations give the matching lower bound (see Theorem~\ref{thrm:approximation-with-quasi-uniform} and~\ref{thrm:approximation-by-codes}). We then pass to almost entropic matroids by rational approximation and free expansion (see Theorem~\ref{thrm:all-polymatroids-are-foldings-of-matroids} and~\ref{thrm:approximation-folding-of-matroids}).

\paragraph{Paper structure}
In Section~\ref{sec:preliminaries}, we collect the required facts about entropy, polymatroids, matroids, and finite posets.
In Section~\ref{sec:closure-systems}, we develop the basic properties of $\mc{M}$-systems for intersection-closed families $\mc{M}$ and establish the existence of asymptotic exponents $\alpha(\mc{M})$ and $\alpha(k,L)$.
In Section~\ref{sec:upper-bound-via-entropic-vectors}, we prove an upper bound on $m(n,\mc{M})$ in terms of a quantity $\eta(\mc{M})$ related to entropic vectors. The proof is simple but already sufficient to deduce Theorem~\ref{thrm:rho-upper-intro}.
In Section~\ref{sec:furedi}, we discuss applications of the relaxed upper bound in terms of $\rho(\mc{M})$ to the study of F\"uredi's conjecture for $(k,L)$-systems.
In Section~\ref{sec:alpha-equals-eta}, we prove the matching entropic lower bound and obtain $\alpha(\mc{M})=\eta(\mc{M})$.
In Section~\ref{sec:lower-bound-via-foldings}, we show how lower bounds can be transferred between closure systems using folding constructions and derive the general lower bound on $\alpha(k,L)$ via rank.
In Section~\ref{sec:lower-bound-ae-matroids}, we combine foldings with almost entropic matroids to prove the lower bound in Theorem~\ref{thrm:main-intro}.
In Section~\ref{sec:upper-bound-ae-matroids}, we approximate compatible almost entropic vectors by rational entropic vectors and prove the reverse inequality in Theorem~\ref{thrm:main-intro}.
In Section~\ref{sec:computability}, we investigate the computability of $\alpha(\mc{M})$.
In Section~\ref{sec:discussion}, we discuss further consequences, related questions, and open problems.

\section{Preliminaries}\label{sec:preliminaries}

We collect the entropy and matroid facts needed later. The poset material will be used for rational approximation in Section~\ref{sec:upper-bound-ae-matroids}.

\subsection{\texorpdfstring{$U$-indexed multipartite families}{U-indexed multipartite families}}

In what follows, it will be convenient to index the parts of a multipartite family by an arbitrary finite set $U$, rather than by $[k]$.

For a finite set $U$, define
\[
    [n]^U=\fbrc{\fbrc{(x_u,u)}{u\in U}}{x_u\in[n]\text{ for every }u\in U}.
\]
For $F\in[n]^U$ and $A\seq U$, define the projection of $F$ onto $A$ by
\[
    \pi_A(F)=\fbrc{(x,u)\in F}{u\in A}\in[n]^A.
\]
For a family $\mc{F}\seq[n]^U$, define
\[
    \pi_A(\mc{F})
    =
    \fbrc{\pi_A(F)}{F\in\mc{F}}.
\]

For $u\in U$, we write $\pi_u(F)\in[n]$ for the value of $F$ in coordinate $u$.

\subsection{Entropy}

\begin{definition}
    For a finite-support random variable $\xi$, its Shannon entropy is
    \[
        \Ent(\xi)=-\sum_{x\in\supp(\xi)}\P(\xi=x)\log_2\P(\xi=x).
    \]
\end{definition}

To work with our bounds, it is convenient to use the language of entropic vectors.
\begin{definition}[\cite{2007MatusApproximation,2008Yeung}]\label{def:entropic-vectors}
    Let $U$ be a finite set. A vector $g\in\R^{2^U}$, viewed as a function $g:2^U\to\R$, is called an \emph{entropic vector} if there exists a tuple of $\lbr{U}$ finite-support random variables $(\xi_u)_{u\in U}$ such that
    \[
        \forall A\seq U\quad g(A)=\Ent((\xi_u)_{u\in A}).
    \]
    The set of all entropic vectors in $\R^{2^U}$ is denoted by $\Gent{U}$. Its topological closure is denoted by $\clGent{U}$, and its elements are called \emph{almost entropic vectors}.
\end{definition}

The region of almost entropic vectors forms a convex cone. In particular, it is independent of the choice of the logarithm base in the definition of entropy.
\begin{proposition}[\cite{1997Zhang}]\label{prop:almost-entropic-vectors-are-cone}
    Let $U$ be a finite set, let $a_1,a_2\ge0$, and let $g_1,g_2\in\clGent{U}$. Then $a_1g_1+a_2g_2\in\clGent{U}$.
\end{proposition}

\subsection{Polymatroids and matroids}

Matroids admit several equivalent definitions; here we use the characterization in terms of their rank functions.

\begin{definition}[\cite{2011Oxley}]
    Let $U$ be a finite set. A \emph{matroid} $N$ on $U$ is a pair $(U,\rk_N)$, where $\rk_N:2^U\to\Z_{\ge0}$ satisfies
    \begin{align*}
        \forall A\seq U\quad & 0 \le \rk_N(A) \le \lbr{A},\\
        \forall A\seq B\seq U\quad & \rk_N(A) \le \rk_N(B),\\
        \forall A,B\seq U\quad & \rk_N(A \cup B) + \rk_N(A \cap B)
        \le \rk_N(A) + \rk_N(B).
    \end{align*}
    The third condition is called \emph{submodularity}. The rank of the matroid $N$ is $\rk(N)=\rk_N(U)$.
\end{definition}

Polymatroids generalize matroids and provide the natural language for entropy vectors. We use the definitions from~\cite{2023Bonin}.
\begin{definition}
\begin{subequations}\label{eq:polymatroid-conditions}
    Let $U$ be a finite set. A pair $(U,g)$ is called a \emph{polymatroid} on $U$ if $g:2^U\to\R$ satisfies
    \begin{align}
        &g(\emptyset)=0;\label{eq:polymatroid-zero-on-empty}\\
        \forall A\seq B\seq U\quad &g(A)\le g(B);\label{eq:polymatroid-monotone}\\
        \forall A,B\seq U\quad &g(A\cup B)+g(A\cap B)\le g(A)+g(B).\label{eq:polymatroid-submodular}
    \end{align}
    It is called a \emph{$1$-polymatroid} if, in addition,
    \begin{equation}
        \forall A\seq U\quad g(A)\le \lbr{A}.\label{eq:polymatroid-one-bounded}
    \end{equation}
\end{subequations}
    An \emph{integer polymatroid} is a polymatroid for which $g$ is integer-valued.
\end{definition}
Thus matroids are precisely the integer $1$-polymatroids.

Each polymatroid naturally induces an intersection-closed family.
\begin{definition}
    Let $(U,g)$ be a polymatroid. Define its family of \emph{flats} by
    \[
        \Flats(g)=\fbrc{A\seq U}{\forall u\in U\setminus A,\ g(A)<g(A\cup\fbr{u})}.
    \]
\end{definition}
\begin{proposition}\label{prop:flats-of-polymatroid-are-intersection-closed}
    If $(U,g)$ is a polymatroid, then $\Flats(g)$ contains $U$ and is closed under intersections.
\end{proposition}
\begin{proof}
    $U$ lies in $\Flats(g)$ since $U\setminus U$ is empty.

    Let $A,B\in\Flats(g)$ and $x\in U\setminus(A\cap B)$. Without loss of generality, assume that $x\notin A$. Then
    \[g(A\cup \fbr{x}) - g(A) > 0.\]
    By submodularity, we have
    \[g((A\cap B)\cup \fbr{x}) - g(A\cap B) \ge g(A\cup \fbr{x}) - g(A) > 0.\]
    Since $x$ is arbitrary, $A\cap B$ lies in $\Flats(g)$.
\end{proof}

Matroids can be used to construct integer polymatroids in the following way.
\begin{proposition}\label{prop:matroid-can-make-integer-polymatroids}
    Let $U$ and $V$ be finite sets, let $(V,\rk_M)$ be a matroid on $V$, and let $f:U\to 2^V$ be an arbitrary map. Define $g:2^U\to \Nz$ by
    \[\forall A\seq U\quad g(A)=\rk_M(f(A)),\]
    where $f(A)=\cup_{u\in A}f(u)$. Then $(U,g)$ is an integer polymatroid.
\end{proposition}
\begin{proof}
    Conditions~\eqref{eq:polymatroid-zero-on-empty},~\eqref{eq:polymatroid-monotone}, and integrality are obvious.

    To check Condition~\eqref{eq:polymatroid-submodular}, fix $A,B\seq U$. By submodularity of $\rk_M$,
    \[\rk_M(f(A)\cup f(B)) + \rk_M(f(A)\cap f(B)) \le \rk_M(f(A)) + \rk_M(f(B)).\]
    Since $f(A\cap B)\seq f(A)\cap f(B)$, we have
    \begin{align*}
        g(A\cup B) + g(A\cap B)&\le \rk_M(f(A)\cup f(B)) + \rk_M(f(A)\cap f(B))\\
        &\le \rk_M(f(A)) + \rk_M(f(B)) = g(A) + g(B).
    \end{align*}
\end{proof}

Conversely, every integer polymatroid arises from a matroid in this way~\cite{1974Helgason,1975McDiarmid,1977Lovasz,1978Nguyen}.
\begin{theorem}[{\cite[Proposition
1.3.9]{1978Nguyen}}]\label{thrm:all-polymatroids-are-foldings-of-matroids}
    Let $U$ be a finite set, let $(U,g)$ be an integer polymatroid. Let $\fbr{T_u}_{u\in U}$ be a family of pairwise disjoint sets such that, for each $u$, $\lbr{T_u}=g(\fbr{u})$. Let $V=\cup_{u\in U} T_u$. Define $f:U\to 2^V$ by $f(u)=T_u$. Define a function $\rk_{M_g}:2^V\to\Nz$ by
    \[\forall X\seq V\quad \rk_{M_g}(X) = \min_{A\seq U}\left(g(A) + \lbr{X\setminus f(A)}\right),\]
    where $f(A)=\cup_{u\in A}f(u)$.

    Then $M_g=(V,\rk_{M_g})$ is a matroid, called a \emph{free expansion} of $g$ by $f$, or a \emph{natural matroid} of $g$. Moreover,
    \[\forall A\seq U\quad g(A)=\rk_{M_g}(f(A)).\]
\end{theorem}
For the properties of the natural matroid of an integer polymatroid, see~\cite{2023Bonin}.

\subsection{Almost entropic polymatroids}\label{subsec:almost-entropic-polymatroids}

By the Shannon-type information inequalities~\cite{1978Fujishige}, every almost entropic vector is a polymatroid.
\begin{proposition}[\cite{1978Fujishige}]\label{prop:almost-entropic-vector-is-polymatroid}
    Let $U$ be a finite set and let $g\in\clGent{U}$. Then $(U,g)$ is a polymatroid.
\end{proposition}
Polymatroids arising from almost entropic vectors are called \emph{almost entropic} polymatroids.

One rich family of almost entropic polymatroids is given by linear matroids.
\begin{proposition}[\cite{1992Matus,2020Studeny}]\label{prop:linear-matroids-are-almost-entropic}
    Let $N$ be a linear matroid. Then $N$ is almost entropic.
\end{proposition}
\begin{proof}
    Since every linear matroid is representable over a finite field~\cite{1957Rado}, and restrictions of almost entropic matroids are almost entropic, it suffices to prove the proposition when $N$ is the linear matroid formed by all vectors in $\F_q^d$, for some integer $d$ and prime power $q$.

    Let $\xi$ be a uniformly random linear map $\F_q^d \to \F_q$, viewed as a $q^d$-dimensional random vector. Then, for every $A$, $H(\xi_A) = \log_2(\lbr{\supp(\xi_A)}) = \log_2(q^{\rk_N(A)}) = \rk_N(A)\cdot \log_2(q)$.
    Hence, by Proposition~\ref{prop:almost-entropic-vectors-are-cone}, $\rk_N$ is almost entropic.
\end{proof}

For an integer polymatroid $(U,g)$, it is easy to check that if its natural matroid $M_g$ is almost entropic, then $g$ is an almost entropic polymatroid. The converse is a deep fact~\cite[Theorem 4]{2007MatusApproximation}.
\begin{theorem}[\cite{2007MatusApproximation}]\label{thrm:natural-matroid-is-almost-entropic}
    Let $(U,g)$ be an integer almost entropic polymatroid. Then its natural matroid $M_g$ is an almost entropic matroid.
\end{theorem}
This theorem allows us to pass from integer almost entropic polymatroids to almost entropic matroids in Section~\ref{sec:upper-bound-ae-matroids}. We include in Appendix~\ref{appendix:natural-matroid-is-almost-entropic} a streamlined adaptation of the proof from~\cite{2007MatusApproximation}, recast in terms of quasi-uniform random variables. Our approach makes the underlying construction more transparent and replaces the more involved asymptotic estimates in the original proof with simpler ones.

\subsection{Posets}

\begin{definition}
    A poset is a pair $(P,\le)$ consisting of a set $P$ and a relation $\le$ on $P$ such that
    \begin{enumerate}
        \item $\forall a\in P,\ a\le a$;
        \item $\forall a,b\in P,\ \text{if } a\le b \text{ and } b\le a\text{, then }a=b$;
        \item $\forall a,b,c\in P,\ \text{if } a\le b \text{ and } b\le c\text{, then }a\le c$.
    \end{enumerate}
\end{definition}

Every intersection-closed family $\mc{M}\seq 2^{U}$ with $U\in\mc{M}$ induces a poset. The following proposition is straightforward to check.
\begin{proposition}\label{prop:closure-system-is-poset}
    Let $U$ be a finite set, and let $\mc{M}\seq 2^{U}$ be an intersection-closed family with $U\in\mc{M}$.
    Then $(\mc{M},\seq)$ is a poset with minimum element $\cap_{A\in \mc{M}}A$ and maximum element $U$.
\end{proposition}

On a finite poset $P$, one can define the M\"obius function~\cite{2011Stanley}.
\begin{definition}
Let $(P,\le)$ be a finite poset. Define the M\"obius function $\mu:P\times P\to \Z$ for every $x,y \in P$ such that $x\le y$ by setting
\[\mu(x,y)=\begin{cases}1,&\text{if } x=y;\\ -\sum_{x\le z < y}\mu(x,z), &\text{if } x < y.\end{cases}\]
\end{definition}

The following proposition is easy to check by induction~\cite[Proposition 3.8.5]{2011Stanley}.
\begin{proposition}[\cite{2011Stanley}]\label{prop:formula-for-mobius-function}
    Let $(P,\le)$ be a finite poset. Then, for every $x,y\in P$ such that $x\le y$, one has
    \[\mu(x,y)=\sum_{t\ge 0}(-1)^t c_t(x,y),\]
    where $c_t(x,y)$ is the number of strict chains from $x$ to $y$ of length $t$, where a strict chain of length $t$ is a sequence $\fbr{a_i}_{i=0}^{t}$ such that $x=a_0 < a_1 < \ldots < a_t = y$.
\end{proposition}

\begin{corollary}\label{corollary:bound-on-mobius-function}
    Let $(P,\le)$ be a finite non-empty poset. Then, for every $x,y\in P$ such that $x\le y$,
    \[|\mu(x,y)|\le 2^{\lbr{P}-1}.\]
\end{corollary}
\begin{proof}
    By Proposition~\ref{prop:formula-for-mobius-function},
    \[|\mu(x,y)|\le \sum_{t\ge 0}c_t(x,y).\]
    Since any strict chain from $x$ to $y$ is ordered and contains the element $x$, the total number of such chains is bounded by $2^{\lbr{P}-1}$.
\end{proof}

\begin{theorem}\label{thrm:poset-inverse-for-R}
    Let $(P,\le)$ be a finite non-empty poset with minimum element $\bot\in P$ and maximum element $\top\in P$. Let $m=\lbr{P}-1$. Define the $m\times m$ matrix $R$, whose rows are indexed by $P\setminus\fbr{\bot}$ and whose columns are indexed by $P\setminus\fbr{\top}$, by setting
    \[\forall a\in P\setminus\fbr{\bot}, \forall b\in P\setminus\fbr{\top}\quad R_{a,b}=\begin{cases}1,&\text{if }a\not\le b;\\ 0,&\text{if } a\le b.\end{cases}\]
    Define the $m\times m$ matrix $W$, whose rows are indexed by $P\setminus\fbr{\top}$ and whose columns are indexed by $P\setminus\fbr{\bot}$, by setting
    \[\forall c\in P\setminus\fbr{\top}, \forall a\in P\setminus\fbr{\bot}\quad W_{c,a}=\begin{cases}-\mu(c,a),&\text{if }c\le a;\\ 0,&\text{if } c\not\le a.\end{cases}\]
    Then $W$ is an inverse of $R$; that is, $WR=\operatorname{Id}$.

    Moreover, $\norminf{W}\le m\cdot 2^m$.
\end{theorem}
\begin{proof}
    Fix $c,b\in P\setminus\fbr{\top}$. Then
    \[(WR)_{c,b}=\sum_{\substack{a\in P\\ c\le a,\ a\not\le b}}-\mu(c,a).\]

    From the definition of the M\"obius function, it follows that
    \[\forall x< y\in P\quad \sum_{z\in P,\ x\le z\le y}\mu(x,z)=0.\]
    Since $c < \top$, we have
    \[\sum_{\substack{a\in P\\ c\le a,\ a\not\le b}}\mu(c,a) + \sum_{\substack{a\in P\\ c\le a,\ a\le b}}\mu(c,a) = \sum_{a\in P,\ c\le a\le \top}\mu(c,a) = 0.\]
    Therefore,
    \[(WR)_{c,b}=\sum_{a\in P,\ c\le a\le b}\mu(c,a)=\begin{cases}1,&\text{if }c=b;\\ 0,&\text{otherwise}.\end{cases}\]

    By Corollary~\ref{corollary:bound-on-mobius-function},
    \[\norminf{W}=\max_{c\in P\setminus\fbr{\top}}\sum_{a\in P\setminus\fbr{\bot}}|W_{c,a}| \le m\cdot 2^m.\]
\end{proof}

\section{Closure systems}\label{sec:closure-systems}

\begin{definition}[\cite{2003Caspard}]
    Let $U$ be a finite set. A family $\mc{M}\seq 2^U$ is called a \emph{closure system} on $U$ if $U\in\mc{M}$ and $A\cap B\in\mc{M}$ for every $A,B\in\mc{M}$.
\end{definition}

A closure system naturally induces a closure operator.
\begin{definition}
    Let $\mc{M}$ be a closure system on $U$. For $A\seq U$, define
    \[
        \cl_{\mc{M}}(A)=\bigcap_{\substack{B\in\mc{M}\\A\seq B}}B.
    \]
\end{definition}
Thus $\cl_{\mc{M}}(A)$ is the smallest member of $\mc{M}$ containing $A$.

It is convenient to allow an arbitrary finite set of coordinate labels rather than only $[k]$. Thus, for a finite set $U$, we work with subfamilies of $[n]^U$, the $\lbr{U}$-fold Cartesian product of $[n]$ with coordinates indexed by $U$.

Recall the definition of $\mc{M}$-systems.
\begin{definition}
    Let $\mc{M}$ be a closure system on $U$. A family $\mc{F}\seq [n]^U$ is called an $\mc{M}$-system if, for every $F,F'\in\mc{F}$,
    \[
        \fbrc{u\in U}{\pi_u(F)=\pi_u(F')}\in\mc{M}.
    \]
\end{definition}

The closure operation gives the following convenient characterization of $\mc{M}$-systems.
\begin{proposition}\label{prop:M-system-via-closure}
    Let $\mc{M}$ be a closure system on $U$ and let $\mc{F}\seq[n]^U$. Then $\mc{F}$ is an $\mc{M}$-system if and only if
    \[
        \forall A\seq U\quad \lbr{\pi_A(\mc{F})}=\lbr{\pi_{\cl_{\mc{M}}(A)}(\mc{F})}.
    \]
\end{proposition}
\begin{proof}
    Suppose that $\mc{F}$ is an $\mc{M}$-system. Fix $A\seq U$. Take $F,F'\in\mc{F}$ such that $\pi_A(F)=\pi_A(F')$. Since
    \[A\seq \fbrc{u\in U}{\pi_u(F)=\pi_u(F')}\in\mc{M},\]
    we have $\cl_{\mc{M}}(A)\seq \fbrc{u\in U}{\pi_u(F)=\pi_u(F')}$ and so $\pi_{\cl_{\mc{M}}(A)}(F)=\pi_{\cl_{\mc{M}}(A)}(F')$. Thus $\pi_A(F)$ uniquely determines $\pi_{\cl_{\mc{M}}(A)}(F)$. Therefore $\lbr{\pi_A(\mc{F})}=\lbr{\pi_{\cl_{\mc{M}}(A)}(\mc{F})}$.

    Conversely, suppose that the projection equality holds for every $A\seq U$. Let $F,F'\in\mc{F}$ and put $A=\fbrc{u\in U}{\pi_u(F)=\pi_u(F')}$. If $A\notin\mc{M}$, then $A\subsetneq\cl_{\mc{M}}(A)$. Since $\pi_A(F)=\pi_A(F')$ and $\pi_{\cl_{\mc{M}}(A)}(F)\ne\pi_{\cl_{\mc{M}}(A)}(F')$, $\pi_A(F)$ does not uniquely determine $\pi_{\cl_{\mc{M}}(A)}(F)$, contradicting $\lbr{\pi_A(\mc{F})}=\lbr{\pi_{\cl_{\mc{M}}(A)}(\mc{F})}$. Hence $A\in\mc{M}$.
\end{proof}

We need the following construction, introduced by R\"odl and Tengan~\cite{2006Rodl}.
\begin{proposition}\label{prop:product-code-mM}
   Let $\mc{M}$ be a closure system on $U$. Then, for all integers $a \ge 1$ and $b \ge 1$,
    \[
        m(a \cdot b,\mc{M}) \ge m(a,\mc{M}) \cdot m(b,\mc{M}).
    \]
\end{proposition}
\begin{proof}
    For completeness, we include the proof. Let $\mc{F}_1 \subseteq [a]^U$, $\mc{F}_2 \subseteq [b]^U$ be $\mc{M}$-systems such that $\lbr{\mc{F}_1} = m(a,\mc{M})$ and $\lbr{\mc{F}_2} = m(b,\mc{M})$.

    Define $\mc{G} \subseteq ([a] \times [b])^U$ as follows:
    \[
        \mc{G} = \fbrc{\fbrc{(\pi_u(F_1), \pi_u(F_2))}{u \in U}}{F_1 \in \mc{F}_1,\ F_2 \in \mc{F}_2}.
    \]
    For any $G,G'\in\mc{G}$, with $G=(F_1,F_2)$ and $G'=(F_1',F_2')$, we have
    \[
        \fbrc{u}{\pi_u(G) = \pi_u(G')} = \fbrc{u}{\pi_u(F_1) = \pi_u(F_1')} \cap \fbrc{u}{\pi_u(F_2) = \pi_u(F_2')},
    \]
    which belongs to $\mc{M}$ since it is closed under intersections. Therefore, $\mc{G}$ is an $\mc{M}$-system and
    \[
        m(a \cdot b, \mc{M}) \ge \lbr{\mc{G}} = \lbr{\mc{F}_1} \cdot \lbr{\mc{F}_2} = m(a,\mc{M}) \cdot m(b,\mc{M}).
    \]
\end{proof}

This product construction implies that the limit $\lim_{n\to\infty}\log_n m(n,\mc{M})$ exists.
\begin{theorem}\label{thrm:alpha-for-mM-exists}
Let $\mc{M}$ be a closure system on $U$. Then the following limit exists:
\[
\alpha(\mc{M}) = \lim_{n \to \infty} \log_n m(n,\mc{M}).
\]
Moreover,
\[
\forall n \ge 1 \quad m(n,\mc{M}) \le n^{\alpha(\mc{M})}.
\]
\end{theorem}
\begin{proof}
    It suffices to show that, for all $a \ge 2$, the following inequality holds:
    \begin{equation}\label{eq:lim-inf-asympt-m}
        \liminf_{n \to \infty} \log_n m(n,\mc{M}) \ge \log_a m(a,\mc{M}).
    \end{equation}

    Fix $a \ge 2$ and $n \ge 2$. Choose $N$ such that $a^N \le n < a^{N+1}$. Then, by Proposition~\ref{prop:product-code-mM},
    \[
        m(n,\mc{M}) \ge m(a^N,\mc{M}) \ge m(a,\mc{M})^N.
    \]
    Also, since $\ln(n) \le (N+1)\ln(a)$, we have
    \[
        \log_n m(n,\mc{M}) \ge \frac{N}{N+1} \log_a m(a,\mc{M}).
    \]
    Since $N \to \infty$ as $n \to \infty$, inequality \eqref{eq:lim-inf-asympt-m} follows. Taking the supremum over $a\ge2$ gives $\alpha(\mc{M})=\sup_{a\ge2}\log_a m(a,\mc{M})$, which also proves the second assertion; the case $n=1$ is immediate.
\end{proof}

Combining this theorem with Theorem~\ref{thrm:kL-to-M} gives the existence of the exponent for $(k,L)$-systems.
\begin{corollary}\label{corollary:alpha-kL-exists}
    Let $k\ge 1$ and $L\seq\fbr{0,\ldots,k-1}$. Then the limit
    \[
        \alpha(k,L)=\lim_{n\to\infty}\log_n m(n,k,L)
    \]
    exists. Moreover,
    \[
        \alpha(k,L)=\max\fbrc{\alpha(\mc{M})}{\begin{array}{l}
        \mc{M}\text{ is a closure system on }[k],\\
        \lbr{A}\in L\text{ for every }A\in\mc{M}\setminus\fbr{[k]}
        \end{array}}.
    \]
\end{corollary}

\section{Upper bound via entropic vectors}\label{sec:upper-bound-via-entropic-vectors}

To prove Theorem~\ref{thrm:main-intro}, it is convenient first to express $\alpha(\mc{M})$ as an optimization problem over the region of almost entropic vectors. We begin with the following notion of compatibility.

\begin{definition}\label{def:compatibility}
    Let $\mc{M}$ be a closure system on $U$. A function $g:2^U\to \R$ is called \emph{compatible with $\mc{M}$} if
    \[
        \forall A\seq U\quad g(A)=g(\cl_{\mc{M}}(A)).
    \]
\end{definition}

We first obtain the following natural upper bound via entropic vectors.
\begin{theorem}\label{thrm:alpha-le-eta}
    Let $\mc{M}$ be a closure system on $U$. Then
    \[\alpha(\mc{M}) \le \eta(\mc{M}),\]
    where
    \[
    \eta(\mc{M})=\max\fbrc{g(U)}{g\in\clGent{U},\ g(A)\le \lbr{A}\text{ for all }A\seq U,\ g\text{ is compatible with }\mc{M}}.
\]
\end{theorem}
The maximum in $\eta(\mc{M})$ exists because the feasible set is closed and bounded.
\begin{proof}
    The proof assigns a random vector to an $\mc{M}$-system. A similar approach was used by Westerb\"ack, Freij-Hollanti, and Hollanti~\cite{2015Westerback} in coding theory.

    Fix $\epsilon > 0$. Then there exist an integer $n\ge 2$ and an $\mc{M}$-system $\mc{F}\seq [n]^U$ such that
    \begin{equation*}
        \log_n(\lbr{\mc{F}})\ge \alpha(\mc{M}) - \epsilon.
    \end{equation*}

    Take a uniformly distributed random vector $\xi_{U}=(\xi_{u})_{u\in U}$ supported on $\mc{F}$, that is,
    \[
    \forall F\in\mc{F}\quad \P(\xi_U = F) =\frac{1}{\lbr{\mc{F}}}.
    \]

    Define the function $g:2^U\to \R$ by
    \[
    \forall A\seq U\quad g(A) = \frac{\Ent\left(\xi_A\right)}{\log_2(n)}.
    \]

    Since $\lbr{\pi_A(\mc{F})} \le n^{\lbr{A}}$, the random vector $\xi_{A}$ takes at most $n^{\lbr{A}}$ values, so $g(A) \le \lbr{A}$. Since $\clGent{U}$ is a convex cone by Proposition~\ref{prop:almost-entropic-vectors-are-cone}, the function $g$ is an almost entropic vector.

    This vector is compatible with $\mc{M}$. Indeed, by Proposition~\ref{prop:M-system-via-closure}, for every $A\seq U$ the projection $\xi_A$ determines $\xi_{\cl_{\mc{M}}(A)}$. Hence
    \(
        \Ent(\xi_A)=\Ent(\xi_{\cl_{\mc{M}}(A)}),
    \)
    so $g(A)=g(\cl_{\mc{M}}(A))$.

    Since $\xi_U$ is uniformly distributed on its support,
    \[\eta(\mc{M}) \ge g(U) = \log_{n}(\lbr{\mc{F}})\ge \alpha(\mc{M}) - \epsilon.\]

    Since $\epsilon>0$ can be chosen arbitrarily small, the claim follows.
\end{proof}

By Proposition~\ref{prop:almost-entropic-vector-is-polymatroid}, every almost entropic vector is a polymatroid. Thus we obtain the following upper bound.
\begin{corollary}\label{corollary:alpha-le-rho}
    Let $\mc{M}$ be a closure system on $U$. Then
    \[
        \alpha(\mc{M})\le\eta(\mc{M})\le\rho(\mc{M}),
    \]
    where
    \begin{align*}
        \rho(\mc{M})
        =\max\Big\{\ g(U)\ \Big|\ & g:2^{U} \to \R,\\
            \forall A\seq U\quad& 0\le g(A)\le \lbr{A};\\
            \forall A\seq B\seq U\quad& g(A)\le g(B);\\
            \forall A,B\seq U\quad& g(A\cup B) + g(A\cap B) \le g(A) + g(B); \\
        \forall A\seq U\quad & g(A)=g(\cl_{\mc{M}}(A))\Big\}.
    \end{align*}
\end{corollary}

The optimization defining $\rho(\mc{M})$ is a finite linear program with integer coefficients and a nonempty bounded feasible set. In particular, $\rho(\mc{M})$ is a computable rational number. Equivalently, the following compact characterization follows directly from the definition of a polymatroid.
\begin{proposition}\label{prop:rho-is-max-polymatroid}
Let $\mc{M}$ be a closure system on $U$. Then
    \[
        \rho(\mc{M})=\max\fbrc{g(U)}{(U,g)\text{ is a $1$-polymatroid compatible with }\mc{M}}.
    \]
\end{proposition}

The polymatroid bound is at least as strong as the rank bound from the introduction.
\begin{proposition}\label{prop:rho-le-than-rk}
    Let $\mc{M}$ be a closure system on $U$. Then
    \[
        \rho(\mc{M}) \le \rk(\mc{M}).
    \]
\end{proposition}
\begin{proof}
    Fix a $1$-polymatroid $(U,g)$ compatible with $\mc{M}$. Let $C \subseteq U$ be such that $\lbr{C} = \rk(\mc{M})$ and $\cl_{\mc{M}}(C) = U$. Then
    \[
    g(U) = g(\cl_{\mc{M}}(C)) = g(C)\le \lbr{C} = \rk(\mc{M}).
    \]

    Therefore $\rho(\mc{M}) \le \rk(\mc{M})$.
\end{proof}

Combining Theorem~\ref{thrm:kL-to-M} and Corollary~\ref{corollary:alpha-le-rho}, we obtain the following corollary.
\begin{corollary}
    Let $k\ge 1$ and $L\seq\fbr{0,\ldots,k-1}$. Then
    \[
        m(n,k,L)=O_{k}(n^{\rho(k,L)}),
    \]
    where
    \[
        \rho(k,L)=\max\fbrc{\rho(\mc{M})}{\begin{array}{l}
        \mc{M}\text{ is a closure system on }[k],\\
        \lbr{A}\in L\text{ for every }A\in\mc{M}\setminus\fbr{[k]}
        \end{array}}.
    \]
\end{corollary}
Together with Corollary~\ref{corollary:alpha-le-rho} and Proposition~\ref{prop:rho-le-than-rk}, this corollary proves Theorem~\ref{thrm:rho-upper-intro}.

\section{Conjecture of F\"uredi}\label{sec:furedi}

For a closure system $\mc{M}$ on $U$, the rank function from the introduction can equivalently be defined by
\[
    \rk(\mc{M})=\min\fbrc{\lbr{C}}{C\seq U,\ \cl_{\mc{M}}(C)=U}.
\]
By Proposition~\ref{prop:rho-le-than-rk}, $\rho(\mc{M})\le\rk(\mc{M})$. As discussed in the introduction, the bound $\alpha(k,L)\le \rk(k,L)$ is not always tight, since the exponent $\alpha(k,L)$ need not be an integer. However, the currently known examples with non-integer exponent do not rule out the following conjecture of F\"uredi~\cite{1991Furedi}.
\begin{conjecture}[\cite{1991Furedi}]\label{conj:Furedi-k-L}
    For every $k\ge 1$ and $L\seq\fbr{0,\ldots,k-1}$,
    \[
        \ceil{\alpha(k,L)}=\rk(k,L).
    \]
\end{conjecture}
F\"uredi stated the conjecture in terms of the growth of $m(n,k,L)$; Corollary~\ref{corollary:alpha-kL-exists} gives the formulation above. This conjecture holds for $k\le 12$~\cite{1996Frankl,2006Tokushige} and for $\rk(k,L)=2$~\cite{1991Furedi}.

A direct generalization of this conjecture to closure systems would assert that $\ceil{\alpha(\mc{M})}=\rk(\mc{M})$. The following examples show that this generalization already fails for rank three.
\begin{proposition}\label{prop:counterexample-3-2}
    There exists a closure system $\mc{M}$ on 6 elements such that $\rk(\mc{M})=3$ and $\alpha(\mc{M})\le 2$.
\end{proposition}
\begin{proof}
    Consider the following closure system $\mc{M}$ on $U = \{1,2,3,4,5,6\}$:
    \begin{gather*}
        \fbr{1, 2, 3, 4, 5, 6},\\
        \fbr{1, 2, 4, 5},
        \fbr{1, 2, 4, 6},
        \fbr{2, 3, 4, 5},\\
        \fbr{1, 2, 4},
        \fbr{1, 3, 4},
        \fbr{2, 4, 5},
        \fbr{2, 5, 6},\\
        \fbr{1, 4},
        \fbr{2, 4},
        \fbr{2, 5},
        \fbr{2, 6},
        \fbr{3, 4},
        \fbr{3, 6},\\
        \fbr{2},
        \fbr{3},
        \fbr{4},
        \fbr{6},\\
        \emptyset.
    \end{gather*}

    It is easy to verify that this family is intersection-closed and $\rk(\mc{M})=3$.

    Let $(U,g)$ be an arbitrary $1$-polymatroid compatible with $\mc{M}$. Summing the following inequalities,
    \begin{alignat*}{2}
        g(\fbr{3})&\  &\le&\ 1;\\
        g(\fbr{6})&\  &\le&\ 1;\\
        g(\fbr{1, 4})&\ \Big(= g(\cl_{\mc{M}}(\fbr{1}))\Big)\ &\le&\ 1;\\
        g(\fbr{2, 5})&\ \Big(= g(\cl_{\mc{M}}(\fbr{5}))\Big)\ &\le&\ 1;\\
        g(\fbr{2,3,4,5})&\  &\le&\ g(\fbr{2}) + g(\fbr{3});\\
        g(\fbr{1,2,4,6})&\  &\le&\ g(\fbr{4}) + g(\fbr{6});\\
        g(\fbr{1,3,4})&+ g(\fbr{4})\  &\le&\ g(\fbr{1,4}) + g(\fbr{3,4});\\
        g(\fbr{1,2,3,4,5,6})&+ g(\fbr{2})\  &\le&\ g(\fbr{1,2,4,6}) + g(\fbr{2,5});\\
        g(\fbr{1,2,3,4,5,6})&+ g(\fbr{3,4})\ &\le&\ g(\fbr{1,3,4}) + g(\fbr{2,3,4,5}).
    \end{alignat*}
    we obtain
    \[
        2 \cdot g(\{1, 2, 3, 4, 5, 6\}) \le 4.
    \]
    Hence $\rho(\mc{M}) \le 2$, and therefore $\alpha(\mc{M}) \le 2$.
\end{proof}

In fact, there exists a closure system $\mc{M}$ for which $\alpha(\mc{M})<\rk(\mc{M})-1$.
\begin{proposition}\label{prop:counterexample-3-19-10}
    There exists a closure system $\mc{M}$ on 10 elements such that $\rk(\mc{M})=3$ and $\alpha(\mc{M})\le \frac{19}{10}$.
\end{proposition}

The construction in Proposition~\ref{prop:counterexample-3-19-10} and its verification are given in Appendix~\ref{appendix:example-10}.

To extend the counterexample to arbitrary rank, we use direct sums.
\begin{definition}
    Let $\mc{M}$ be a closure system on $U$ and let $\mc{N}$ be a closure system on $V$. Choose disjoint copies $U',V'$ of the ground sets, together with isomorphic closure systems $\mc{M}'$ on $U'$ and $\mc{N}'$ on $V'$. Define
    \[
        \mc{M}\oplus\mc{N}=\fbrc{A\cup B}{A\in\mc{M}',\ B\in\mc{N}'}.
    \]
\end{definition}

The direct sum $\mc{M}\oplus\mc{N}$ is again a closure system and satisfies the following relations with its summands.
\begin{proposition}\label{prop:rk-and-m-of-direct-sum}
    Let $\mc{M}$ be a closure system on $U$ and let $\mc{N}$ be a closure system on $V$. Then
    \[
        \rk(\mc{M}\oplus\mc{N})=\rk(\mc{M})+\rk(\mc{N})
    \]
    and, for every $n\ge 1$,
    \[
        m(n,\mc{M}\oplus\mc{N})=m(n,\mc{M})\cdot m(n,\mc{N}).
    \]
    Consequently, $\alpha(\mc{M}\oplus\mc{N})=\alpha(\mc{M})+\alpha(\mc{N})$.
\end{proposition}
\begin{proof}
    Without loss of generality, we may assume that $U$ is disjoint from $V$ and that $\mc{M}\oplus\mc{N}$ is defined on $U\cup V$.

    The expression for the rank follows from the fact that, for $A\seq U$ and $B\seq V$,
    \[\cl_{\mc{M}\oplus\mc{N}}(A\cup B)=\cl_{\mc{M}}(A) \cup \cl_{\mc{N}}(B).\]

    Fix $n\ge 1$. Let $\mc{F}\seq [n]^U$ be an $\mc{M}$-system, and let $\mc{G}\seq [n]^V$ be an $\mc{N}$-system. Define $\mc{F}'\seq [n]^{U\cup V}$ by
    \[\mc{F}'=\fbrc{((\pi_u(F))_{u\in U}, (\pi_v(G))_{v\in V})}{F\in\mc{F},\ G\in\mc{G}}.\]
    It is straightforward to check that $\mc{F}'$ is an $\mc{M}\oplus\mc{N}$-system. Thus
    \[m(n,\mc{M}\oplus\mc{N}) \ge m(n,\mc{M})\cdot m(n,\mc{N}).\]

    Let $\mc{F}\seq [n]^{U\cup V}$ be an $\mc{M}\oplus\mc{N}$-system with $\lbr{\mc{F}}=m(n,\mc{M}\oplus\mc{N})$. It is straightforward to check that $\pi_{U}(\mc{F})$ is an $\mc{M}$-system and $\pi_{V}(\mc{F})$ is an $\mc{N}$-system. Thus
    \[m(n,\mc{M}\oplus\mc{N})=\lbr{\mc{F}}\le \lbr{\pi_{U}(\mc{F})}\cdot \lbr{\pi_{V}(\mc{F})} \le m(n,\mc{M})\cdot m(n,\mc{N}).\]
    The identity for $\alpha$ follows from Theorem~\ref{thrm:alpha-for-mM-exists}.
\end{proof}

\begin{corollary}\label{corollary:counterexample-Furedi-S-all-r}
    For every integer $r\ge 3$ there exists a closure system $\mc{M}$ such that $\rk(\mc{M})=r$ and
    \[\alpha(\mc{M})\le \frac45r,\qquad \alpha(\mc{M})<r-1.\]
\end{corollary}
\begin{proof}
    Write $r=3q+t$, where $q\ge1$ and $t\in\{0,1,2\}$. Let $\mc{K}$ be the direct sum of $q$ copies of the closure system from Proposition~\ref{prop:counterexample-3-19-10} and $2^{[t]}$. By Proposition~\ref{prop:rk-and-m-of-direct-sum}, $\rk(\mc{K})=r$ and
    \[\alpha(\mc{K})\le \frac{19}{10}q+t\le \frac45r,\qquad
      \alpha(\mc{K})\le r-\frac{11}{10}q<r-1.\]
\end{proof}

\section{Lower bound via entropic vectors}\label{sec:alpha-equals-eta}

To prove the lower bound $\alpha(\mc{M})\ge \eta(\mc{M})$, we first show that \emph{almost entropic} vectors compatible with a closure system $\mc{M}$ can be approximated by \emph{entropic} vectors compatible with $\mc{M}$. We need the following lemma, which shows that if conditional entropy is small, then one can add a random variable of small entropy so that the conditional entropy becomes zero.

\begin{lemma}\label{lemma:small-fix-of-conditional-entropy}
    There is a universal function $f$ with $f(0)=0$ and $f(c)\to0$ as $c\to0$ such that, for any finite-support random variables $X$ and $Y$, one can adjoin a finite-support random variable $R$ to $Y$ so that
    \[\Ent(X|(Y,R))=0,\qquad \Ent(R)\le f(\Ent(X|Y)).\]
\end{lemma}
\begin{proof}

Let $c = \Ent(X|Y)$. If $c=0$, then we can take $R$ to be constant with $\Ent(R)=0$. Henceforth, assume that $c > 0$.

For each $y\in \supp(Y)$ and each $x \in \supp(X)$, let
\[p(x|y)=\mathbb{P}(X=x|Y=y).\]

For each $y\in \supp(Y)$, define
\[\supp(X | Y=y) = \fbrc{x\in \supp(X)}{p(x|y) > 0}.\]
Set $s_y = \lbr{\supp(X | Y=y)}$.

For each $y$, fix an ordering of $\supp(X | Y=y)$, $x_{y,1},\ldots,x_{y,s_y}$, such that
\[\forall i\le s_y-1\quad p(x_{y,i}|y)\ge p(x_{y,i+1}|y).\]

Define an integer-valued random variable $R$ by
\[X=x_{Y,R}.\]

Let $L=\floor{\log_2(R)}$. Fix $y\in \supp(Y)$. Since the sequence $\br{p(x_{y,i}|y)}_{i\in[s_y]}$ is non-increasing and its sum is equal to one, we have
\[\forall i\in [s_y]\quad p(x_{y,i}|y) \le \frac{1}{i}.\]
Hence
\[\forall i\in [s_y]\quad \log_2(i) \le -\log_2(p(x_{y,i}|y)).\]
Thus
\[\mathbb{E}(L|Y=y)=\sum_{i\in [s_y]}p(x_{y,i}|y)\cdot\floor{\log_2(i)}\le -\sum_{i\in [s_y]}p(x_{y,i}|y)\cdot\log_2(p(x_{y,i}|y)) = \Ent(X|Y=y).\]
Therefore,
\[\mathbb{E}(L) \le \Ent(X|Y)=c.\]

Recall that a geometric distribution on $\mathbb{Z}_{\ge 0}$ with parameter $p$ takes the value $l$ with probability $p(1-p)^l$. Among non-negative integer-valued random variables with a given mean, the geometric distribution has the largest entropy, as shown in~\cite[Theorem 5.8]{ConradMaximumEntropy}. Therefore, $\Ent(L)$ can be bounded from above by the entropy of a geometric distribution with the required mean. For completeness, we provide the full proof.

Let $p=\frac{1}{1+c}$. Jensen's inequality for $\log_2$ yields
\[\sum_{l\in \supp(L)}\mathbb{P}(L=l)\log_2\br{\frac{p(1-p)^l}{\mathbb{P}(L=l)}}\le \log_2\br{\sum_{l\ge 0}p(1-p)^l}=0.\]
Thus
\begin{equation}\label{eq:ent-L-estimate-by-mean}
\begin{aligned}
    \Ent(L) \le -\sum_{l\ge 0}\mathbb{P}(L=l)\log_2(p(1-p)^l)&=-\log_2(p) - \log_2(1-p)\mathbb{E}(L)\\
    &\le(1+c)\log_2(1+c)-c\log_2(c).
\end{aligned}
\end{equation}

If $L=l$, then $R$ can take only values in $\fbr{2^l,\ldots,2^{l+1}-1}$, so there are at most $2^{l}$ possible values. Hence $\Ent(R|L=l)\le \log_2(2^l)=l$, and
\[\Ent(R|L)=\mathbb{E}_{l\sim L}\Ent(R|L=l)\le \mathbb{E}_{l\sim L}(l) = \mathbb{E}(L).\]
Then, by \eqref{eq:ent-L-estimate-by-mean},
\[\Ent(R) \le \Ent(R|L) + \Ent(L) \le c + (1+c)\log_2(1+c)-c\log_2(c).\]

Since the right-hand side tends to zero as $c\rightarrow 0$, the required claim follows.

\end{proof}

By definition, almost entropic vectors are limits of entropic vectors. The next theorem shows that, when the limit is compatible with $\mc{M}$, the approximating vectors can also be chosen compatible with $\mc{M}$.

\begin{theorem}\label{thrm:approximation-with-compatible-with-M}
    Let $\mc{M}$ be a closure system on $U$ and let $g\in\clGent{U}$ be an almost entropic vector compatible with $\mc{M}$. Then there exists a sequence $g_i\in\Gent{U}$ of entropic vectors compatible with $\mc{M}$ such that
    \[\forall A\seq U\quad \lim_{i\rightarrow+\infty} g_i(A) = g(A).\]
\end{theorem}
\begin{proof}
    We first approximate $g$ by an entropic vector, then use Lemma~\ref{lemma:small-fix-of-conditional-entropy} to enforce compatibility at an entropy cost tending to zero.

    Fix $\epsilon > 0$. Since $g$ is an almost entropic vector, there is a finite-support random vector $\xi_U = (\xi_u)_{u\in U}$ such that
    \[\forall A\seq U\quad |g(A) - \Ent(\xi_A)| < \epsilon.\]

    Since $g$ is compatible with $\mc{M}$, for every $A\seq U$ we have
    \[|\Ent(\xi_{\cl_{\mc{M}}(A)})-\Ent(\xi_{A})| \le |\Ent(\xi_{\cl_{\mc{M}}(A)})-g(\cl_{\mc{M}}(A))| + |g(\cl_{\mc{M}}(A)) - g(A)| + |g(A)-\Ent(\xi_{A})| \le 2\epsilon.\]

    For every $A\seq U$ with $A\not\subseteq\cl_{\mc{M}}(\emptyset)$, take $R_A$ from Lemma~\ref{lemma:small-fix-of-conditional-entropy} with $X=\xi_{\cl_{\mc{M}}(A)}$ and $Y=\xi_A$. For every $A\subseteq\cl_{\mc{M}}(\emptyset)$, take $R_{A}=(\xi_u)_{u\in A}$.
    Let $T=(R_A)_{A\seq U}$.
    Take $\wt{\xi}_U=(\wt{\xi}_u)_{u\in U}$, where
    \[\wt{\xi}_u=\begin{cases}\text{constant},\quad&\text{if }u\in\cl_{\mc{M}}(\emptyset);\\(\xi_u,T),\quad&\text{otherwise}.\end{cases}\]

    For every $A\seq U$ with $A\not\subseteq\cl_{\mc{M}}(\emptyset)$, by the properties of $R_A$, $\Ent(\xi_{\cl_{\mc{M}}(A)}|\wt{\xi}_A)=0$, so
    \[\Ent(\wt{\xi}_A)=\Ent(\wt{\xi}_{\cl_{\mc{M}}(A)}).\]
    For every $A\subseteq\cl_{\mc{M}}(\emptyset)$,
    \[\Ent(\wt{\xi}_{A})=0=\Ent(\wt{\xi}_{\cl_{\mc{M}}(A)}).\]
    Thus $(\Ent(\wt{\xi}_A))_{A\seq U}$ is an entropic vector compatible with $\mc{M}$.

    Moreover,
    \[\forall A\seq U\quad |\Ent(\xi_A) - \Ent(\wt{\xi}_A)|\le \Ent(T) + \Ent(\xi_{\cl_{\mc{M}}(\emptyset)})\le \sum_{A\seq U}\Ent(R_A) + \Ent(\xi_{\cl_{\mc{M}}(\emptyset)}).\]
    The right-hand side tends to zero as $\epsilon$ tends to zero. Hence
    \[\forall A\seq U\quad |g(A) - \Ent(\wt{\xi}_A)| < \epsilon',\]
    where $\epsilon'\rightarrow 0$ as $\epsilon \rightarrow 0$, which completes the proof.
\end{proof}

The main tool for proving the lower bound $\alpha(\mc{M})\ge\eta(\mc{M})$ is the fact that almost entropic vectors can be approximated by entropic vectors induced by a particularly regular class of random vectors, called \emph{quasi-uniform} random vectors.

\begin{definition}
    Let $U$ be a finite set. A finite-support random vector $\xi_U=(\xi_u)_{u\in U}$ is called quasi-uniform if, for every $A\seq U$, the random vector $\xi_A=(\xi_u)_{u\in A}$ is uniformly distributed on its support.
\end{definition}

The following theorem adapts the approximation in~\cite[Theorem 16.22]{2008Yeung} to preserve compatibility with $\mc{M}$; we give the construction directly using sequences of prescribed frequencies.
\begin{theorem}\label{thrm:approximation-with-quasi-uniform}
    Let $\mc{M}$ be a closure system on $U$.
    Let $g\in \clGent{U}$ be a non-zero almost entropic vector compatible with $\mc{M}$. Then, for every $\epsilon>0$, there exist a constant $r > 0$ and a quasi-uniform random vector $\eta_U=(\eta_u)_{u\in U}$ compatible with $\mc{M}$ such that
    \[\forall A\seq U\quad \lbr{\frac{1}{r}\cdot \Ent(\eta_A) - g(A)} < \epsilon.\]
\end{theorem}
\begin{proof}
    By Theorem~\ref{thrm:approximation-with-compatible-with-M}, it suffices to consider an entropic vector $g$ realized by a finite-support random vector $\xi_U=(\xi_u)_{u\in U}$. We may assume that each $\xi_u$ takes values in $[n]$ for some $n\ge2$. Write $p_U(G)=\P(\xi_U=G)>0$ for $G\in\supp(\xi_U)$. Continuity of $-p\log_2p$ allows us to approximate these probabilities by positive rational probabilities on the same support, changing all marginal entropies arbitrarily little and preserving compatibility. We therefore assume from now on that all $p_U(G)$ are rational.

    Choose $q\ge 1$ such that, for every $G\in\supp(\xi_U)$, the number $q\cdot p_U(G)$ is an integer.

    Take an arbitrary integer $m\ge 1$. Let $r=q\cdot m$. Let
    \[\mc{F}=\fbrc{F=(F_1,\ldots,F_r)\in \br{[n]^U}^r}{\forall G\in \supp(\xi_U)\ \lbr{\fbrc{i\in[r]}{F_i = G}}=p_U(G)\cdot r},\]
    that is, the set of sequences of length $r$ of elements of $[n]^U$ such that each $G\in \supp(\xi_U)$ occurs exactly $p_U(G)\cdot r$ times.

     For $A\seq U$ and $G\in\supp(\xi_A)\seq [n]^A$, set
    $p_{A}(G)=\P(\xi_A = G)$. Then
    \begin{equation}\label{eq:formula-p-A-G}
    \forall A\seq U,\ \forall G\in\supp(\xi_A)\quad p_{A}(G) = \sum_{\substack{G'\in \supp(\xi_U)\\ \pi_A(G')=G}}p_U(G').
    \end{equation}

    For $A\seq U$, set
    \[\pi_A(\mc{F}) = \fbrc{(\pi_A(F_i))_{i\in[r]}}{F\in\mc{F}}\seq \br{[n]^A}^r.\]
    Then, by \eqref{eq:formula-p-A-G},
    \begin{equation}\label{eq:formula-for-pi-A-F}
    \pi_A(\mc{F})=\fbrc{F=(F_1,\ldots,F_r)\in \br{[n]^A}^r}{\forall G\in\supp(\xi_A)\ \lbr{\fbrc{i\in[r]}{F_i = G}}=p_{A}(G)\cdot r}.
    \end{equation}

    Since $\xi$ is compatible with $\mc{M}$, for every $A\seq U$ we have $\lbr{\supp(\xi_A)}=\lbr{\supp(\xi_{\cl_{\mc{M}}(A)})}$, and hence
    \begin{equation}\label{eq:quasi-uniform-approximation-mcF-is-compatible-with-M}
        \forall A\seq U\quad \lbr{\pi_A(\mc{F})}=\lbr{\pi_{\cl_{\mc{M}}(A)}(\mc{F})}.
    \end{equation}

    By~\eqref{eq:formula-for-pi-A-F}, any two elements $X,Y\in\pi_A(\mc{F})$ differ by a permutation of the $r$ positions. The same permutation gives a bijection between the fibres $\mc{F}[X]$ and $\mc{F}[Y]$, where $\mc{F}[X]=\fbrc{F\in\mc{F}}{\pi_A(F)=X}$. Hence
    \begin{equation}\label{eq:quasi-uniform-approximation-equal-U-extension}\forall A\seq U\ \forall X,Y\in \pi_A(\mc{F})\quad \lbr{\mc{F}[X]}=\lbr{\mc{F}[Y]}.\end{equation}

    Take a uniformly distributed random vector $\eta_U=(\eta_u)_{u\in U}$ supported on $\mc{F}$, viewed as a subset of $[n^r]^U$, that is, on
    \[\wt{\mc{F}}=\fbrc{\br{\br{\pi_u(F_i)}_{i\in[r]}}_{u\in U}}{F\in \mc{F}}.\]

    From \eqref{eq:quasi-uniform-approximation-equal-U-extension}, it follows that $\eta$ is quasi-uniform. From \eqref{eq:quasi-uniform-approximation-mcF-is-compatible-with-M}, it follows that $\eta$ is compatible with $\mc{M}$.

    By Stirling's formula, $n!\sim \sqrt{2\pi n}\br{\frac{n}{e}}^n$, and hence the entropy of $\eta_A$ is equal to
    \begin{gather*}
        \forall A\seq U\quad \ln(2)\cdot\Ent(\eta_A)=\ln(\lbr{\pi_A(\mc{F})}) = \ln\br{\frac{r!}{\prod_{G\in \supp(\xi_A)}(p_A(G)\cdot r)!}}\\
        =\ln(r!) - \sum_{G\in \supp(\xi_A)}\ln((p_A(G)\cdot r)!)\\
        =r\ln(r) - r - \sum_{G\in \supp(\xi_A)}\bigg((p_A(G)\cdot r)\ln(p_A(G)\cdot r)
        - (p_A(G)\cdot r)\bigg) + O(\ln(r))\\
        =-\sum_{G\in \supp(\xi_A)}(p_A(G)\cdot r)\ln(p_A(G)) + O(\ln(r))\\
        =r\cdot\ln(2)\Ent(\xi_A) + O(\ln(r)).
    \end{gather*}

    Hence
    \[\forall A\seq U\quad \lbr{\frac{1}{r}\Ent(\eta_A)-\Ent(\xi_A)}\le O\br{\frac{\ln(r)}{r}}.\]
    For sufficiently large $r$, the right-hand side is smaller than $\epsilon$, which gives the required claim.
\end{proof}

The support of a quasi-uniform random vector compatible with $\mc{M}$ is an $\mc{M}$-system in which all fibres of each projection have the same size. This gives the following approximation.
\begin{theorem}\label{thrm:approximation-by-codes}
    Let $\mc{M}$ be a closure system on $U$.
    Let $g\in \clGent{U}$ be a non-zero almost entropic vector compatible with $\mc{M}$. Then, for every $\epsilon>0$, there exist $n \ge 2$ and an $\mc{M}$-system $\mc{F}\seq [n]^U$ such that
    \[\forall A\seq U\quad |c\cdot \log_n(\lbr{\pi_{A}(\mc{F})}) - g(A)| \le \epsilon,\]
    where $c=\max_{u\in U}g(\fbr{u})$.
    Moreover,
    \begin{equation}\label{eq:quasi-uniform-property-mcF}
        \forall A\seq U,\ \forall X\in \pi_A(\mc{F})\quad \lbr{\mc{F}[X]} = \frac{\lbr{\mc{F}}}{\lbr{\pi_A(\mc{F})}},
    \end{equation}
    where, for $A\seq U$ and $X\in\pi_A(\mc F)$, $\mc{F}[X]=\fbrc{F\in\mc{F}}{\pi_A(F)=X}$.
\end{theorem}

\begin{proof}
    Let $\epsilon'=\frac{\epsilon}{|U|+1}$. By Theorem~\ref{thrm:approximation-with-quasi-uniform}, there exist a real $r > 0$ and a quasi-uniform random vector $\eta_U=(\eta_u)_{u\in U}$ compatible with $\mc{M}$ such that
    \begin{equation}\label{eq:quasi-uniform-distance-small}
    \forall A\seq U\quad \lbr{\frac{1}{r}\Ent(\eta_A) - g(A)}\le \epsilon'.
    \end{equation}
    Let $n=\max_{u\in U} 2^{\Ent(\eta_u)}$. Since $g$ is non-zero, we may assume that $n\ge 2$. We may also assume that the support of each $\eta_u$ is a subset of $[n]$.

    By \eqref{eq:quasi-uniform-distance-small},
    \[\lbr{\frac{\log_2(n)}{r}-c}<\epsilon'.\]

    Then
    \[\forall A\seq U\quad \lbr{\frac{c}{\log_2(n)}\Ent(\eta_A) - \frac{1}{r}\Ent(\eta_A)} \le \Ent(\eta_A)\cdot\frac{\epsilon'}{\log_2(n)} \le \lbr{U}\cdot \epsilon'.\]

    Hence
    \begin{equation}\label{eq:updated-c-is-small-distance}
        \forall A\seq U\quad \lbr{\frac{c}{\log_2(n)}\Ent(\eta_A) - g(A)}\le (\lbr{U}+1)\epsilon'.
    \end{equation}

    Let $\mc{F}=\supp(\eta_U)\seq [n]^U$. Since $\eta_U$ is compatible with $\mc{M}$, the family $\mc{F}$ is an $\mc{M}$-system. Since $\eta_U$ is quasi-uniform, \eqref{eq:quasi-uniform-property-mcF} holds. By \eqref{eq:updated-c-is-small-distance},
    \[\forall A\seq U\quad \lbr{c\cdot\log_n(\lbr{\pi_A(\mc{F})}) - g(A)}\le (\lbr{U}+1)\epsilon'=\epsilon.\]
\end{proof}

\begin{corollary}\label{corollary:alpha-equals-eta}
    Let $\mc{M}$ be a closure system on $U$. Then
    \[\alpha(\mc{M})=\eta(\mc{M}).\]
\end{corollary}
\begin{proof}
    The upper bound was proved in Theorem~\ref{thrm:alpha-le-eta}. We prove the lower bound.

    Suppose that $\eta(\mc{M})$ is attained at $g:2^U\to\R$. We may assume that $g(U) > 0$, since otherwise the lower bound is trivial. Then $c=\max_{u\in U}g(\fbr{u}) > 0$ and $c\le 1$. Fix $0 < \epsilon < g(U)$. Since $g$ is compatible with $\mc{M}$, by Theorem~\ref{thrm:approximation-by-codes} there exist $n\ge 2$ and an $\mc{M}$-system $\mc{F}\seq [n]^U$ such that
    \[\log_n(\lbr{\mc{F}}) \ge \frac{g(U) - \epsilon}{c} \ge g(U) - \epsilon = \eta(\mc{M}) - \epsilon.\]
    Since $\epsilon > 0$ can be chosen arbitrarily small, we obtain
    \[\alpha(\mc{M}) \ge \eta(\mc{M}).\]
\end{proof}

\section{Lower bound via foldings}\label{sec:lower-bound-via-foldings}

\begin{definition}\label{def:folding-closure-systems}
    Let $\mc{M}$ be a closure system on $U$ and let $\mc{N}$ be a closure system on $V$. A map $f:U\to 2^V$ is called a \emph{folding of $\mc{N}$ into $\mc{M}$} if
    \begin{enumerate}
        \item $\cl_{\mc{N}}(f(U))=V$; \label{eq:folding-cond-U-eq-V}
        \item for every $A\seq U$, $\cl_{\mc{N}}(f(\cl_{\mc{M}}(A)))=\cl_{\mc{N}}(f(A))$;\label{eq:folding-preserve-cl}
    \end{enumerate}
    where, for $A\seq U$, $f(A)=\bigcup_{u\in A}f(u)$.

    The rank of the folding is
    \[
        \rk(f)=\max\br{1,\max_{u\in U}\lbr{f(u)}}.
    \]
\end{definition}

Folding constructions transfer lower bounds from one closure system to another.
\begin{theorem}\label{thrm:lower-bound-on-m-via-folding}
    Let $\mc{M}$ be a closure system on $U$, let $\mc{N}$ be a closure system on $V$, and let $f:U\to 2^V$ be a folding of $\mc{N}$ into $\mc{M}$. Then
    \[
        \forall n\ge 1\quad m(n^{\rk(f)},\mc{M})\ge m(n,\mc{N}).
    \]
\end{theorem}

\begin{proof}
    For fixed $n \ge 1$, take $\mc{F} \subseteq [n]^V$ such that $\mc{F}$ is an $\mc{N}$-system and $\lbr{\mc{F}} = m(n,\mc{N})$.

    Let $a = n^{\rk(f)}$. Then, for any $u \in U$, we have $a \ge n^{\lbr{f(u)}}\ge \lbr{\pi_{f(u)}(\mc{F})}$.

    For each $u\in U$, take an injective map $g_u:\pi_{f(u)}(\mc{F})\to[a]$.

    Construct the following family $\mc{G}$ on $[a]^U$: for each $F \in \mc{F}$, include $G_F\in [a]^U$ in $\mc{G}$, where
    \[
        \forall u \in U \quad \pi_u(G_F) = g_u(\pi_{f(u)}(F)).
    \]

    For distinct $F, F' \in \mc{F}$, define:
    \[
        A = \fbrc{u \in U}{\pi_u(G_F) = \pi_{u}(G_{F'})}, \quad X = \fbrc{v \in V}{\pi_v(F) = \pi_v(F')}.
    \]
    Then,
    \begin{equation}\label{eq:thrm-lower-bound-def-A-for-G}
        A = \fbrc{u \in U}{f(u) \subseteq X}.
    \end{equation}

    Since $\mc{F}$ is an $\mc{N}$-system, we have $X\in \mc{N}$. We know that $f(A)\seq X$, so $\cl_{\mc{N}}(f(A))\seq X$ and, by folding Condition~\ref{eq:folding-preserve-cl}, $f(\cl_{\mc{M}}(A))\seq X$. Therefore, by \eqref{eq:thrm-lower-bound-def-A-for-G}, we have $A=\cl_{\mc{M}}(A)$, and hence $\mc{G}$ is an $\mc{M}$-system. Moreover, by folding Condition~\ref{eq:folding-cond-U-eq-V}, $A\neq U$, since $\cl_{\mc{N}}(f(A))\seq X\in \mc{N}\setminus \fbr{V}$, and thus all $G_F$ are distinct. Hence,
    \[
        m(n^{\rk(f)}, \mc{M}) \ge m(n,\mc{N}).
    \]
\end{proof}

\begin{definition}\label{def:folding-rank}
    Define the \emph{folding rank} of $\mc{M}$ in $\mc{N}$ by
    \[
        \frk{\mc{M}}{\mc{N}}=\min\fbrc{\rk(f)}{f\text{ is a folding of }\mc{N}\text{ into }\mc{M}},
    \]
    with value $+\infty$ if no such folding exists.

    For a matroid $N$, we write
    $\frk{\mc{M}}{N}=\frk{\mc{M}}{\Flats(N)}$.
\end{definition}

Taking logarithms in Theorem~\ref{thrm:lower-bound-on-m-via-folding} gives the following corollary.

\begin{corollary}\label{corollary:folding-lower-bound-alpha}
    For closure systems $\mc{M}$ and $\mc{N}$,
    \[
        \alpha(\mc{M})\ge\frac{\alpha(\mc{N})}{\frk{\mc{M}}{\mc{N}}}.
    \]
\end{corollary}

The simplest nontrivial closure systems with known exponent are $2^{[s]}$. Indeed, $m(n,2^{[s]})=n^s$ and $\alpha(2^{[s]})=s$. Folding rank in these closure systems admits the following concise description.
\begin{proposition}\label{prop:iff-folding-2k}
    Let $s\ge 0$ be an integer and let $\mc{M}$ be a closure system on $U$. Then
    \[
        \frk{\mc{M}}{2^{[s]}}
        =\min\fbrc{a\ge 1}{\exists A_1,\ldots,A_s\in\mc{M}\setminus\fbr{U},\
        \forall u\in U\ \lbr{\fbrc{i\in[s]}{u\notin A_i}}\le a}.
    \]
\end{proposition}
\begin{proof}
    Fix $a\ge 1$. Suppose that we have a folding $f: U \to 2^{[s]}$ with $\rk(f)\le a$. For all $i \in [s]$, define
    \[
        A_i = \fbrc{u}{i \not\in f(u)}.
    \]
    Since $f(U) = [s]$, we have $A_i \ne U$ for every $i$. Since $\cl_{2^{[s]}}(f(A_i))=f(A_i)$, from folding Condition~\ref{eq:folding-preserve-cl} we have
    \begin{equation}\label{eq:f-cl-equals-f-A}
    f(\cl_{\mc{M}}(A_i)) = f(A_i).
    \end{equation}
    If $A_i\neq \cl_{\mc{M}}(A_i)$, then there exists $u\in \cl_{\mc{M}}(A_i)$ such that $i\in f(u)$, which contradicts \eqref{eq:f-cl-equals-f-A}. Thus $A_i=\cl_{\mc{M}}(A_i)\in \mc{M}$ for every $i\in [s]$.

    For every $u\in U$, the definition of $A_i$ gives
    \[
        \lbr{\fbrc{i}{u\notin A_i}}=\lbr{f(u)}\le\rk(f)\le a.
    \]

    Conversely, suppose that we are given $A_1, \ldots, A_s \in \mc{M} \setminus \{U\}$. Define $f: U \to 2^{[s]}$ by
    \[
        \forall u\in U\quad f(u) = \fbrc{i}{u \not\in A_i}.
    \]
    Since each $A_i\neq U$, we have $f(U)=[s]$. Take $B\seq U$. Suppose that $f(B)\neq f(\cl_{\mc{M}}(B))$. Then there exists $i$ such that $B\seq A_i$ but $\cl_{\mc{M}}(B) \not \seq A_i$, which contradicts the fact that $\cl_{\mc{M}}(B)$ is the smallest element of $\mc{M}$ that contains $B$. Hence $f$ is indeed a folding of $2^{[s]}$ into $\mc{M}$. Moreover,
    \[
        \rk(f) = \max\br{1,\max_{u \in U} \lbr{f(u)}} = \max\br{1,\max_{u \in U} \lbr{\fbrc{i}{u \not\in A_i}}} \le a.
    \]
\end{proof}

As an immediate consequence, we obtain the following general lower bound.
\begin{corollary}\label{corollary:rank-lower-bound-M}
    Let $\mc{M}$ be a closure system on $U$. Let $k=\lbr{U}$ and let $r=\rk(\mc{M})\ge 1$. Then
    \[
        \alpha(\mc{M})\ge\frac{k}{k-r+1}.
    \]
\end{corollary}
\begin{proof}
    Let $s = \binom{k}{r-1}$. For each $B\in \binom{U}{r-1}$, that is, each subset of $U$ of size $r-1$, define
    \[
        A_B = \cl_{\mc{M}}(B).
    \]

    Since $\lbr{B}=r-1 < \rk(\mc{M})$, we have $A_B \ne U$ for all $B\in \binom{U}{r-1}$.

    Fix an element $u \in U$. Then
    \[
        \lbr{\fbrc{B \in \binom{U}{r-1}}{u \not\in A_B}} \le \lbr{\fbrc{B \in \binom{U}{r-1}}{u \not\in B}} = \binom{k - 1}{r - 1}.
    \]

    Therefore, by Proposition~\ref{prop:iff-folding-2k}, $\frk{\mc{M}}{2^{[s]}}\le \binom{k - 1}{r - 1}$. By Corollary~\ref{corollary:folding-lower-bound-alpha},
    \[
        \alpha(\mc{M})\ge \frac{\binom{k}{r - 1}}{\binom{k - 1}{r - 1}} = \frac{k}{k-r+1}.
    \]
\end{proof}

Corollaries~\ref{corollary:rank-lower-bound-M} and~\ref{corollary:alpha-kL-exists} yield the following bound, which was proved for $r=2$ by F\"uredi~\cite{1991Furedi}.

\begin{corollary}\label{corollary:rank-lower-bound-kL}
    Let $k\ge 1$ and $L\seq \fbr{0,1,\ldots,k-1}$. Let $r=\rk(k,L)\ge 1$. Then
    \[
        \alpha(k,L)\ge\frac{k}{k-r+1}.
    \]
\end{corollary}

\section{Lower bound via almost entropic matroids}\label{sec:lower-bound-ae-matroids}

Corollaries~\ref{corollary:alpha-equals-eta} and~\ref{corollary:folding-lower-bound-alpha} give the following proposition.
\begin{proposition}\label{prop:eta-folding-lower-bound}
    Let $\mc{M}$ and $\mc{N}$ be closure systems on $U$ and $V$, respectively. Then
    \[
        \eta(\mc{M})\ge\frac{\eta(\mc{N})}{\frk{\mc{M}}{\mc{N}}}.
    \]
\end{proposition}

We next record the characterization of almost entropic matroids needed for the Theorem~\ref{thrm:main-intro}.
\begin{proposition}\label{prop:matroid-rank-function-has-unique-solution}
    Let $N=(V,\rk_N)$ be a matroid and let $g$ be a $1$-polymatroid compatible with $\Flats(N)$ such that $g(V)=\rk(N)$. Then $g=\rk_N$.
\end{proposition}
\begin{proof}
    Fix $A \subseteq V$. Let $l = \rk_N(A)$. Choose $a_1, \ldots, a_l \in A$ such that $A \subseteq \cl_N(\{a_1, \ldots, a_l\})$. Then, from Condition~\eqref{eq:polymatroid-one-bounded} and compatibility, we have
    \[
        g(A) = g(\cl_N(A)) = g(\cl_N(\{a_1, \ldots, a_l\})) = g(\{a_1, \ldots, a_l\}) \le l = \rk_N(A).
    \]
    Let $s = \rk_N(V) - \rk_N(A)$. Choose $b_1, \ldots, b_s \in V$ such that $\cl_N(A \cup \{b_1, \ldots, b_s\}) = V$. Then
    \[
        \rk_N(V) = g(V) = g(A \cup \{b_1, \ldots, b_s\}) \le g(A) + g(\{b_1, \ldots, b_s\}) \le \rk_N(A) + s = \rk_N(V),
    \]
    which implies $g(A) = \rk_N(A)$.
\end{proof}

\begin{corollary}\label{corollary:matroid-eta-characterization}
    A matroid $N$ on $V$ is almost entropic if and only if
    \[
        \eta(\Flats(N))=\rk(N).
    \]
\end{corollary}
\begin{proof}
    If $N$ is almost entropic, its rank function is feasible in the definition of $\eta(\Flats(N))$. Since $\rk(\Flats(N))=\rk(N)$, equality follows from Proposition~\ref{prop:rho-le-than-rk}.

    Conversely, if equality holds, a maximizing almost entropic $1$-polymatroid $g$ satisfies $g(V)=\rk(N)$ and is compatible with $\Flats(N)$. Proposition~\ref{prop:matroid-rank-function-has-unique-solution} gives $g=\rk_N$.
\end{proof}

By Propositions~\ref{prop:rho-is-max-polymatroid} and~\ref{prop:matroid-rank-function-has-unique-solution}, every matroid $N$ satisfies $\rho(\Flats(N))=\rk(N)$. The V\'amos matroid, which is not almost entropic~\cite{2018Matus}, therefore gives a strict gap between $\alpha$ and $\rho$.
\begin{corollary}\label{corollary:alpha-neq-rho}
    Let $N$ be the V\'amos matroid. Then
    \[\alpha(\Flats(N)) < \rho(\Flats(N))=\rk(N)=4.\]
\end{corollary}

\begin{corollary}\label{corollary:eta-ge-etamat}
    Let $\mc{M}$ be a closure system. Then
    \[
        \eta(\mc{M})\ge\etamat(\mc{M}).
    \]
\end{corollary}
\begin{proof}
    Let $N$ be an almost entropic matroid. By Proposition~\ref{prop:eta-folding-lower-bound} and Corollary~\ref{corollary:matroid-eta-characterization},
    \[
        \eta(\mc{M})\ge\frac{\eta(\Flats(N))}{\frk{\mc{M}}{\Flats(N)}}=\frac{\rk(N)}{\frk{\mc{M}}{N}}.
    \]
    Taking the supremum over $N$ proves the claim.
\end{proof}
Together with Corollary~\ref{corollary:alpha-equals-eta}, this proves the lower bound in Theorem~\ref{thrm:main-intro}.

Foldings of matroids give the following characterization of $\rho(\mc{M})$, which justifies viewing it as a relaxation of $\etamat(\mc{M})$.
\begin{theorem}\label{thrm:rho-is-a-folding}
    Let $\mc{M}$ be a closure system on $U$. Then
    \begin{equation*}
        \rho(\mc{M})=\max\fbrc{\frac{\rk(N)}{\frk{\mc{M}}{N}}}{N\text{ is a matroid}}.
    \end{equation*}
\end{theorem}
\begin{proof}
    Consider a matroid $N$ on $V$ and a folding $f:U \to 2^V$ of $\Flats(N)$ into $\mc{M}$. By Proposition~\ref{prop:matroid-can-make-integer-polymatroids} and the definition of $\rk(f)$, the function $g:2^U\to \R$ defined by
    \[
        \forall A\seq U\quad g(A)=\frac{\rk_N(\cup_{u\in A}f(u))}{\rk(f)}
    \]
    is a $1$-polymatroid. Folding Condition~\ref{eq:folding-preserve-cl} shows that $g$ is compatible with $\mc{M}$, while folding Condition~\ref{eq:folding-cond-U-eq-V} gives $g(U)=\frac{\rk(N)}{\rk(f)}$.
    Therefore, by Proposition~\ref{prop:rho-is-max-polymatroid}, $\rho(\mc{M}) \ge g(U)=\frac{\rk(N)}{\rk(f)}$.

    For the reverse inequality, since $\rho(\mc{M})$ is the optimum of a linear program with integer coefficients, there exists a rational-valued $1$-polymatroid $g$ compatible with $\mc{M}$ such that $g(U)=\rho(\mc{M})$ (cf. Proposition~\ref{prop:rho-is-max-polymatroid}). If $U=\emptyset$ or $\max_{u\in U}g(\fbr{u})=0$, then $g(U)=0$, and the required inequality is immediate. Suppose therefore that $\max_{u\in U}g(\fbr{u})>0$. By the optimality of $g(U)$, we must have $\max_{u\in U}g(\fbr{u})=1$. Choose an integer $K\ge 1$ such that $K\cdot g$ is an integer polymatroid. By Theorem~\ref{thrm:all-polymatroids-are-foldings-of-matroids}, there exist a set $V$, a function $f:U\to 2^V$, and a matroid $N$ on $V$ that is the free expansion of $K\cdot g$ by $f$. Since $K\cdot g$ is compatible with $\mc{M}$, it is straightforward to verify that $f$ is a folding of $\Flats(N)$ into $\mc{M}$. Moreover, $\rk(f)=\max_{u\in U}K\cdot g(\fbr{u})=K$. Therefore,
    \[\rho(\mc{M})=g(U)=\frac{\rk(N)}{K}\le \frac{\rk(N)}{\frk{\mc{M}}{N}}.\]
\end{proof}

\section{Upper bound via almost entropic matroids}\label{sec:upper-bound-ae-matroids}

The following theorem shows that every almost entropic vector compatible with $\mc{M}$ can be approximated by a rational-valued entropic vector compatible with $\mc{M}$.
\begin{theorem}\label{thrm:approximation-rational-vectors}
    Let $\epsilon\in (0,1)$ be real, let $\mc{M}$ be a closure system on $U$, and let $g\in\clGent{U}$ be an almost entropic vector compatible with $\mc{M}$. Then there exists a rational-valued entropic vector $h:2^U \to \Q$ compatible with $\mc{M}$ such that
    \[
    \forall A\seq U\quad \lbr{h(A) - g(A)} \le \epsilon.
    \]
    Moreover, there exists an integer $K$ such that $1\le K \le \frac{1}{\epsilon}\cdot \lbr{\mc{M}}^2 2^{\lbr{\mc{M}}}$ and
    \[\forall A\seq U\quad K\cdot h(A)\in \Z.\]
\end{theorem}
\begin{proof}
    The proof is analogous to the argument in~\cite{2007MatusApproximation} showing that every almost entropic vector can be approximated by rational-valued entropic vectors. Here, however, we also establish a bound on the denominator of the rational vector.

    We may assume that $\lbr{\mc{M}}\ge 2$, since otherwise $\mc{M}=\fbr{U}$ and $g$ is the zero vector.

    Let $\epsilon'=\frac{\epsilon}{2}$. Since $g$ is almost entropic, Theorem~\ref{thrm:approximation-with-compatible-with-M} implies that there exists an entropic vector $g'\in \Gent{U}$ compatible with $\mc{M}$ such that
    \begin{equation}\label{eq:g-prime-approx-g}
        \forall A\seq U\quad |g'(A) - g(A)|\le \epsilon'.
    \end{equation}

    For $B\in \mc{M}\setminus\fbr{U}$, define a function $r_B:2^U\to \Nz$ by
    \[\forall A\seq U\quad r_B(A)=\begin{cases}1,&\text{if }A\not\subseteq B;\\ 0,&\text{otherwise}.\end{cases}\]
    For every real $a\ge 0$, the vector $a\cdot r_B$ is entropic. Indeed, one can take a random variable $\eta$ with entropy $a$ and define a random vector $(\xi_u)_{u\in U}$ by
    \[\forall u\in U\quad \xi_u=\begin{cases}\eta,&\text{if }u\not\in B;\\ \text{constant},&\text{otherwise}.\end{cases}\]

    Moreover, $r_B$ is compatible with $\mc{M}$, since, for every $A\seq U$, if $A\seq B$, then $\cl_{\mc{M}}(A) \seq B$.

    Define $\bot=\cl_{\mc{M}}(\emptyset)$ and $m=\lbr{\mc{M}} - 1$. Define the $m\times m$ matrix $R$, whose rows are indexed by $A\in\mc{M}\setminus\fbr{\bot}$ and whose columns are indexed by $B\in \mc{M}\setminus \fbr{U}$, as follows:
    \[\forall A\in\mc{M}\setminus\fbr{\bot}, \forall B\in \mc{M}\setminus \fbr{U}\quad R_{A,B}=r_B(A).\]
    By Theorem~\ref{thrm:poset-inverse-for-R}, the matrix $R$ is invertible and
    \begin{equation}\label{eq:R-inverse-norm-is-small}
        \norminf{R^{-1}} \le m\cdot 2^m.
    \end{equation}

    Let $\delta=\frac{\epsilon'}{m}$. Define
    \[h_0 = g' + \sum_{B\in \mc{M}\setminus \fbr{U}}\frac{\delta}{2}\cdot r_B.\]

    Let $K=\ceil{\frac{1}{\delta}\cdot \norminf{R^{-1}}}$.

    Let $c_{\bot}=0$. For every $A\in \mc{M}\setminus\fbr{\bot}$, choose an integer $c_A$ such that
    \[\lbr{\frac{c_A}{K}-h_0(A)}\le \frac{1}{2K}.\]

    Define a function $h:2^U\to \Q$ by
    \[\forall A\seq U\quad h(A)=\frac{c_{\cl_{\mc{M}}(A)}}{K}.\]

    Define a vector $v:\mc{M}\setminus\fbr{\bot}\to \R$ by
    \[\forall A\in \mc{M}\setminus\fbr{\bot}\quad v(A) = h(A) - h_0(A).\]
    Define a vector $\theta:\mc{M}\setminus\fbr{U}\to \R$ by
    \[\theta = R^{-1}v.\]
    Then
    \begin{equation}\label{eq:theta-norm-is-small}
    \norminf{\theta}\le \norminf{R^{-1}}\norminf{v}\le \frac{\norminf{R^{-1}}}{2K} \le \frac{\delta}{2}.
    \end{equation}

    If we naturally extend the domain of $v$ to $2^U$ by setting $v(A)=v(\cl_{\mc{M}}(A))$ for $A\seq U$ and $v(\bot)=0$, then
    \begin{equation}\label{eq:expression-for-h}
    h = h_0 + R\theta = g' + \sum_{B\in \mc{M}\setminus\fbr{U}}\br{\frac{\delta}{2}+\theta_B}\cdot r_B.
    \end{equation}

    It follows from \eqref{eq:theta-norm-is-small} that the coefficient of each $r_B$ in this sum is nonnegative. Since a finite sum of entropic vectors is entropic~\cite[Lemma~15.3]{2008Yeung}, $h$ is an entropic vector. Moreover, by definition, $h$ is compatible with $\mc{M}$.

    Furthermore, for every $A\seq U$, equations~\eqref{eq:g-prime-approx-g} and~\eqref{eq:expression-for-h} give
    \begin{align*}
        |g(A)-h(A)|
        &\le |g(A)-g'(A)|+|g'(A)-h(A)|\\
        &\le \epsilon'
        +\sum_{B\in \mc{M}\setminus\fbr{U}}
        \lbr{\frac{\delta}{2}+\theta_B}\\
        &\le \epsilon'+m\cdot\delta
        =2\epsilon'=\epsilon.
    \end{align*}

    By \eqref{eq:R-inverse-norm-is-small},
    \[K\le \ceil{\frac{1}{\delta}\cdot m\cdot 2^m},\]
    and hence
    \[K\le 1+\frac{1}{\epsilon'}\cdot m^2\cdot 2^m\le \frac{1}{\epsilon}\cdot \lbr{\mc{M}}^2\cdot 2^{\lbr{\mc{M}}}.\]
\end{proof}

For later use in Section~\ref{sec:computability}, we state the following theorem with an explicit error term and a bound on $\lbr{V}$.
\begin{theorem}\label{thrm:approximation-folding-of-matroids}
    Let $\epsilon\in (0,1)$ be real, and let $\mc{M}$ be a closure system on $U$. Then there exists an almost entropic matroid $N$ on a finite set $V$ such that
    \[\frac{\rk(N)}{\frk{\mc{M}}{N}} \ge \frac{\eta(\mc{M}) - \epsilon}{1+\epsilon},\]
    and
    \[\lbr{V} \le \lbr{U}\cdot\frac{1}{\epsilon}\cdot\lbr{\mc{M}}^2 \cdot 2^{\lbr{\mc{M}}+1}.\]
\end{theorem}
\begin{proof}
    If $U=\emptyset$, take $N$ to be the empty matroid; both conclusions are immediate. Assume that $U\ne\emptyset$.

    Take an almost entropic $1$-polymatroid $g:2^U\to \R$ compatible with $\mc{M}$ such that $g(U)=\eta(\mc{M})$. By Theorem~\ref{thrm:approximation-rational-vectors}, there exist an almost entropic polymatroid $h:2^U \to \Q$ compatible with $\mc{M}$, an integer $1 \le K \le \frac{1}{\epsilon}\cdot \lbr{\mc{M}}^2\cdot 2^{\lbr{\mc{M}}}$ such that $K\cdot h$ is an integer polymatroid, $h(U) \ge g(U)-\epsilon=\eta(\mc{M})-\epsilon$ and
    \[c = \max_{u\in U}h(\fbr{u}) \le 1 + \epsilon.\]

    Since $K\cdot h$ is an integer polymatroid, by Theorem~\ref{thrm:all-polymatroids-are-foldings-of-matroids}, there exist a set $V$, a function $f:U\to 2^V$, and a matroid $N$ on $V$ that is the free expansion of $K\cdot h$ by $f$. By Theorem~\ref{thrm:natural-matroid-is-almost-entropic}, since $K\cdot h$ is an almost entropic polymatroid, $N$, as the natural matroid of $K\cdot h$, is an almost entropic matroid.

    By the definition of the free expansion,
    \[\forall A\seq U\quad \rk_N(f(A))=K\cdot h(A)=K\cdot h(\cl_{\mc{M}}(A))=\rk_N(f(\cl_{\mc{M}}(A))),\]
    where $f(A)=\cup_{u\in A}f(u)$. Therefore, folding Condition~\ref{eq:folding-preserve-cl} holds. Folding Condition~\ref{eq:folding-cond-U-eq-V} follows from the definition of the free expansion. Therefore, $f$ is a folding of $\Flats(N)$ into $\mc{M}$.

    By the definition of $f$,
    \[\rk(f) = \max\br{1,\ \max_{u\in U}K\cdot h(\fbr{u})} \le (1+\epsilon)\cdot K.\]

    Therefore,
    \[\frac{\rk(N)}{\frk{\mc{M}}{N}}\ge \frac{\rk(N)}{\rk(f)}=\frac{K\cdot h(U)}{\rk(f)}\ge \frac{\eta(\mc{M})-\epsilon}{1+\epsilon}.\]

    Also,
    \[\lbr{V} \le \lbr{U}\cdot \max_{u\in U}K\cdot h(\fbr{u})\le \lbr{U} \cdot 2\cdot K.\]
\end{proof}

\begin{corollary}\label{corollary:eta-eq-etamat}
    Let $\mc{M}$ be a closure system. Then
    \[
        \eta(\mc{M})=\etamat(\mc{M}).
    \]
\end{corollary}
\begin{proof}
    Corollary~\ref{corollary:eta-ge-etamat} gives
    \[
        \eta(\mc{M})\ge\etamat(\mc{M}).
    \]
    For the reverse inequality, Theorem~\ref{thrm:approximation-folding-of-matroids} gives, for every $\epsilon\in(0,1)$, an almost entropic matroid $N$ such that
    \[
        \etamat(\mc{M})\ge\frac{\rk(N)}{\frk{\mc{M}}{N}}
        \ge\frac{\eta(\mc{M})-\epsilon}{1+\epsilon}.
    \]
    Letting $\epsilon\to0$ proves the claim.
\end{proof}
From Corollary~\ref{corollary:alpha-equals-eta} we know that $\alpha(\mc{M})=\eta(\mc{M})$. Combined with Corollary~\ref{corollary:eta-eq-etamat} this proves Theorem~\ref{thrm:main-intro}.

\section{\texorpdfstring{Computability of $\alpha(\mc{M})$}{Computability of alpha(M)}}\label{sec:computability}

Even assuming that $\alpha(\mc{M})$ is always rational, there is no algorithm that outputs its exact value for every $\mc{M}$. This follows from the fact that, by the result of~\cite{2025Yashfe}, the set of almost entropic integer polymatroids is undecidable, and so is the set of almost entropic matroids by Theorem~\ref{thrm:natural-matroid-is-almost-entropic}. Therefore, for a matroid $N$, it is impossible to decide algorithmically whether the equality $\alpha(\Flats(N))=\rk(N)$ holds, and consequently, there is no algorithm that outputs $\alpha(\Flats(N))$ exactly.

This raises the following weaker question: can $\alpha(\mc{M})$ be effectively approximated?
\begin{conjecture}\label{conj:computable-approximation-alpha}
    There exists a Turing machine which, given a rational $\epsilon > 0$, a finite set $U$, and a closure system $\mc{M}$ on $U$, returns a rational number $\alpha'_\epsilon$ such that
    \[|\alpha'_\epsilon - \alpha(\mc{M})| \le \epsilon.\]
\end{conjecture}

By Corollary~\ref{corollary:alpha-equals-eta} and Theorem~\ref{thrm:approximation-folding-of-matroids}, this conjecture is equivalent to the following conjecture concerning the set of almost entropic matroids.

\begin{conjecture}\label{conj:co-enumerable-almost-entropic-matroid}
    The set of almost entropic matroids is co-recursively enumerable, that is, there exists a Turing machine which, given a matroid $N=(V,\rk_N)$, halts exactly when $N$ is not an almost entropic matroid.
\end{conjecture}

\begin{theorem}
    Conjecture~\ref{conj:computable-approximation-alpha} is equivalent to Conjecture~\ref{conj:co-enumerable-almost-entropic-matroid}.
\end{theorem}
\begin{proof}
    Suppose that Conjecture~\ref{conj:computable-approximation-alpha} holds. We construct a Turing machine for Conjecture~\ref{conj:co-enumerable-almost-entropic-matroid}. Let a matroid $N$ be given. For each integer $n\ge 1$, the machine will successively compute approximations $\alpha'_{\epsilon}$ to $\alpha(\Flats(N))$ with accuracy $\epsilon=\frac{1}{n}$. If at some point we obtain an $\alpha'_\epsilon$ such that $\rk(N) > \alpha'_{\epsilon} + \epsilon$, then the machine halts. If $N$ is an almost entropic matroid, then $\alpha(\Flats(N))=\rk(N)$, and the machine never halts. If $N$ is not almost entropic, then, once $\epsilon < \frac{\rk(N) - \alpha(\Flats(N))}{2}$, the machine will halt.

    Suppose that Conjecture~\ref{conj:co-enumerable-almost-entropic-matroid} holds. We now construct an approximation algorithm for $\alpha(\mc{M})$.

    Let $\epsilon\in(0,1)$ be rational, let $U$ be a finite set, and let $\mc{M}$ be a closure system on $U$. Let $\epsilon'=\frac{\epsilon}{\lbr{U} + 2}$. By Corollary~\ref{corollary:alpha-equals-eta} and Theorem~\ref{thrm:approximation-folding-of-matroids}, there exist a set $V$ of size
    \begin{equation}\label{eq:bound-on-V}
    \lbr{V} \le \lbr{U}\cdot\frac{1}{\epsilon'}\cdot\lbr{\mc{M}}^2 \cdot 2^{\lbr{\mc{M}}+1}
    \end{equation}
    and an almost entropic matroid $\wt{N}$ on $V$ such that
    \begin{equation*}
    \frac{\rk(\wt{N})}{\frk{\mc{M}}{\wt{N}}} \ge \frac{\alpha(\mc{M}) - \epsilon'}{1+\epsilon'}.
    \end{equation*}

    We will enumerate all nonnegative integers $s$ not exceeding the right-hand side of inequality~\eqref{eq:bound-on-V} and all matroids $N$ on the ground set $[s]$ (there are finitely many of them), and compute $\frk{\mc{M}}{N}$ for each such matroid. This is possible because, for each such $N$, there are only finitely many foldings of $\Flats(N)$ into $\mc{M}$. For each matroid $N$, run in parallel the Turing machine $M_1$ from Conjecture~\ref{conj:co-enumerable-almost-entropic-matroid} and the Turing machine $M_2$, which successively enumerates integers $n\ge 2$, computes the value $m(n,\Flats(N))$ by a finite exhaustive search, and halts if $\log_n(m(n,\Flats(N))) \ge \rk(N) - \epsilon'$.

    At least one of the machines $M_1$ and $M_2$ is guaranteed to halt in finite time, since in the case of almost entropic matroids the machine $M_2$ halts by Theorem~\ref{thrm:alpha-for-mM-exists}, while in the case of matroids that are not almost entropic, at least the machine $M_1$ halts.

    After carrying out this procedure for every matroid $N$, return $\alpha'_\epsilon$ as the maximum of the values $\frac{\rk(N)-\epsilon'}{\frk{\mc{M}}{N}}$ over all matroids $N$ for which $M_2$ is the first machine to halt. Then, by the construction of the machine $M_2$ and Theorem~\ref{thrm:lower-bound-on-m-via-folding}, for every matroid $N$ taken into account in the definition of $\alpha'_\epsilon$, we have
    \[\alpha(\mc{M})\ge \frac{\alpha(\Flats(N))}{\frk{\mc{M}}{N}}\ge \frac{\rk(N) - \epsilon'}{\frk{\mc{M}}{N}}\]
    and therefore $\alpha(\mc{M})\ge \alpha'_\epsilon$.

    On the other hand, among the matroids considered above, there is an almost entropic matroid isomorphic to $\wt{N}$. For this matroid, the machine $M_1$ never halts, whereas $M_2$ is the first machine to halt, and therefore
    \[\alpha'_\epsilon\ge \frac{\rk(\wt{N})-\epsilon'}{\frk{\mc{M}}{\wt{N}}} \ge  \frac{\alpha(\mc{M}) - \epsilon'}{1+\epsilon'} - \frac{\epsilon'}{\frk{\mc{M}}{\wt{N}}} \ge \alpha(\mc{M}) - \epsilon'\cdot \alpha(\mc{M}) - 2\epsilon' \ge \alpha(\mc{M}) - \epsilon.\]
\end{proof}

\section{Discussion and open questions}\label{sec:discussion}

\subsection{\texorpdfstring{Rationality of $\alpha(\mc{M})$}{Rationality of alpha(M)}}

It is natural to ask whether $\alpha(\mc{M})$ is always rational. The following conjecture strengthens this by asserting that the supremum in Theorem~\ref{thrm:main-intro} is attained.
\begin{conjecture}\label{conj:alpha-equals-folding}
    Let $\mc{M}$ be a closure system on $U$. Then there exists an almost entropic matroid $N$ such that
    \[
        \alpha(\mc{M})=\frac{\rk(N)}{\frk{\mc{M}}{N}}.
    \]
\end{conjecture}

By Corollary~\ref{corollary:alpha-equals-eta} and the proof of Theorem~\ref{thrm:approximation-folding-of-matroids}, this admits an equivalent formulation in terms of almost entropic polymatroids.
\begin{conjecture}\label{conj:eta-attained-at-rational-valued}
    Let $\mc{M}$ be a closure system on $U$. Then there exists a rational-valued almost entropic $1$-polymatroid $g:2^U\to\Q$ compatible with $\mc{M}$ such that
    \[
        \eta(\mc{M})=g(U).
    \]
\end{conjecture}

\subsection{Perfect matroid designs}

There is a simple upper bound $\rk(k,L)\le \lbr{L}$. In the case of equality, the following conjecture strengthens Conjecture~\ref{conj:Furedi-k-L}.
\begin{conjecture}\label{conj:mnkl-full-rank-is-L}
    Let $k\ge 1$ and $L\seq\fbr{0,\ldots,k-1}$ be such that $\rk(k,L)=\lbr{L}$. Then
    \[
        \alpha(k,L)=\lbr{L}.
    \]
\end{conjecture}

Frankl, as noted in~\cite{1992Deza}, observed the following connection between matroids and $(k,L)$-systems.
\begin{theorem}[Frankl]\label{thrm:frankl-pmd}
    Let $k\ge 1$ and $L\seq\fbr{0,\ldots,k-1}$. Then $\rk(k,L)=\lbr{L}$ if and only if there exists a matroid $N$ on $[k]$ such that, for every flat $A\in\Flats(N)$, $\lbr{A}\in L\cup \fbr{k}$ and $\rk(N)=\lbr{L}$.
\end{theorem}

The matroids arising in this theorem are particular cases of perfect matroid designs.
\begin{definition}[\cite{1992Deza}]
    A matroid is called a \emph{perfect matroid design} if any two flats of the same rank have the same cardinality.
\end{definition}

Corollaries~\ref{corollary:alpha-equals-eta} and~\ref{corollary:matroid-eta-characterization} motivate the following strengthening of Conjecture~\ref{conj:mnkl-full-rank-is-L}.
\begin{conjecture}\label{conj:PMD-almost-entropic}
    Every perfect matroid design is almost entropic.
\end{conjecture}

Some nontrivial perfect matroid designs arise from Steiner systems. The following proposition is straightforward to prove by checking the axioms for flats of a matroid~\cite[Theorem~39.10]{2003Schrijver}.
\begin{proposition}\label{prop:PMD-from-steiner}
    Let $n\ge k\ge t \ge 1$ be integers. Let $\mc{F}\seq \binom{[n]}{k}$ be a family such that each $X\in \binom{[n]}{t}$ lies in exactly one element of $\mc{F}$. Then
    \[\mc{M}=\mc{F}\cup \bigcup_{l\le t-1}\binom{[n]}{l}\cup \fbr{[n]}\]
    is the family of flats of some matroid $N_{\mc{F}}$ on $[n]$. Furthermore, $N_{\mc{F}}$ is a perfect matroid design.
\end{proposition}

When $k=3$ and $t=2$, a Steiner system exists if and only if $n\equiv 1\pmod 6$ or $n\equiv 3\pmod 6$~\cite{2008Lindner}. This suggests the following conjecture.
\begin{conjecture}\label{conj:alpha-k-0-1-3}
    Let $k\ge 4$ be an integer such that $k\equiv 1\pmod 6$ or $k\equiv 3\pmod 6$. Then
    \[\alpha(k,\fbr{0,1,3})=3.\]
\end{conjecture}
In~\cite{2006Rodl}, it was proved that, for such $k$, $\alpha(k,\fbr{0,1,3})\ge 2 + \epsilon_{k}$ for some $\epsilon_{k} > 0$.

Conjecture~\ref{conj:alpha-k-0-1-3} holds for $k=2^d-1$ (from projective geometries), $k=3^d$ (from affine geometries), and $k=19$~\cite{1980Brouwer}. Indeed, for each such $k$, there exists a Steiner system $\mc{F}\seq \binom{[k]}{3}$ with $t=2$ such that the matroid $N_{\mc{F}}$ from Proposition~\ref{prop:PMD-from-steiner} is linear. Linear matroids are almost entropic~\cite{2018Matus}, and hence, by Theorem~\ref{thrm:main-intro},
$\alpha(k,\fbr{0,1,3})\ge\alpha(\Flats(N_{\mc{F}}))=\rk(N_{\mc{F}})=3$.
However, this proof scheme does not always apply: it was shown in~\cite{1973Kelly, 1980Limbos} that, for $k=13$ and every Steiner system $\mc{F}\seq \binom{[k]}{3}$ with $t=2$, the matroid $N_{\mc{F}}$ is not linear.

\subsection{\texorpdfstring{Relations between $\eta(\mc{M})$, $\rho(\mc{M})$, and $\rk(\mc{M})$}{Relations between eta(M), rho(M), and rk(M)}}

In view of Corollary~\ref{corollary:alpha-equals-eta} and Proposition~\ref{prop:rho-le-than-rk}, it is interesting to understand the relationships between $\eta(\mc{M})=\alpha(\mc{M})$, $\rho(\mc{M})$, and $\rk(\mc{M})$.
\begin{definition}
    For real $k\ge 0,r\ge 0,\rho\ge 0$, define
    \[\rho_{\rk}(k,r)=\min\fbrc{\rho(\mc{M})}{\lbr{U}\le k,\ \mc{M}\text{ is a closure system on }U,\ \rk(\mc{M})\ge r}.\]
    \[\eta_{\rho}(k,\rho)=\min\fbrc{\eta(\mc{M})}{\lbr{U}\le k,\ \mc{M}\text{ is a closure system on }U,\ \rho(\mc{M})\ge \rho}.\]

    With the convention that the minimum over the empty set is equal to $+\infty$.
\end{definition}

We conjecture that the following strengthening of Corollary~\ref{corollary:counterexample-Furedi-S-all-r} holds.
\begin{conjecture}\label{conj:const-rank-exists-1-eps-mM}
    For every $r \ge 1$,
    \[\rho_{\rk}(k,r) = 1+o_{r}(1)\quad\text{as }k\rightarrow+\infty.\]
\end{conjecture}

From Corollary~\ref{corollary:rank-lower-bound-M}, it follows that $\rho_{\rk}(k,r)\ge \frac{k}{k-r+1}$ for $1\le r \le k$. It is interesting to know whether the order of this bound is sharp.

\begin{conjecture}\label{conj:ak-rank-is-theta}
    For every real $a\in (0,1)$,
    \[\rho_{\rk}(k,a\cdot k) = \Theta_a(1)\quad\text{as }k\rightarrow+\infty.\]
\end{conjecture}

It is also interesting to understand how strong the dependence between $\rho$ and $\eta$ can be.
\begin{conjecture}\label{conj:const-rho-exists-1-eps-mM}
    For every $\rho \ge 1$,
    \[\eta_{\rho}(k,\rho) = 1+o_{\rho}(1)\quad\text{as }k\rightarrow+\infty.\]
\end{conjecture}


\section{Declaration of generative AI and AI-assisted technologies in the writing process}

During the preparation of this work, the authors used ChatGPT to improve the readability and language of the manuscript and to identify and correct typographical, grammatical, formatting, and other non-substantive errors. ChatGPT helped write the Proof Method paragraph in the Introduction and other text fragments. ChatGPT identified and provided a reference to the work~\cite{2015Westerback}, which shows how codes can generate polymatroids; as a result, the proof of the upper bound $\alpha(\mc{M})\le\rho(\mc{M})$ was simplified from a combinatorial proof to an entropy-based proof. ChatGPT pointed out that an almost entropic vector can be approximated by quasi-uniform random vectors, as proved in~\cite[Theorem~16.22]{2008Yeung}. This result made it possible to prove the equality $\alpha(\mc{M})=\eta(\mc{M})$. ChatGPT provided proofs of Lemma~\ref{lemma:small-fix-of-conditional-entropy} and Theorem~\ref{thrm:approximation-with-compatible-with-M}. ChatGPT identified how Mat\'u\v{s}'s proof could be modified to obtain a bound on the denominator of the rational approximation in Theorem~\ref{thrm:approximation-rational-vectors}, and also provided a proof of Theorem~\ref{thrm:poset-inverse-for-R}. ChatGPT helped work out the technical details of the new proof of Theorem~\ref{thrm:natural-matroid-is-almost-entropic} in Appendix~\ref{appendix:natural-matroid-is-almost-entropic}, starting from the idea that one needs to use random compression of codes arising from a quasi-uniform approximation of an almost entropic vector.

ChatGPT service was used to improve proofs and to obtain proofs of several relatively straightforward propositions, such as Proposition~\ref{prop:M-system-via-closure}, Proposition~\ref{prop:rho-le-than-rk}, Corollary~\ref{corollary:matroid-eta-characterization}, Corollary~\ref{corollary:eta-ge-etamat}, and Corollary~\ref{corollary:eta-eq-etamat}.

After using this service, the authors reviewed and edited the resulting content as needed and take full responsibility for the content of the published article.

\appendix

\section{Proof of Theorem~\ref{thrm:natural-matroid-is-almost-entropic}}\label{appendix:natural-matroid-is-almost-entropic}

Let $\fbr{T_u}_{u\in U}$ be a family of pairwise disjoint sets such that, for each $u\in U$, $\lbr{T_u}=g(\fbr{u})$. Let $V=\cup_{u\in U} T_u$. Define $f:U\to 2^V$ by $f(u)=T_u$. Define a function $\rk_N:2^V\to\Nz$ by
\[\forall C\seq V\quad \rk_N(C) = \min_{A\seq U}\left(g(A) + \lbr{C\setminus f(A)}\right),\]
where $f(A)=\cup_{u\in A}f(u)$.

Then, by Theorem~\ref{thrm:all-polymatroids-are-foldings-of-matroids}, $N=(V,\rk_N)$ is a matroid, and we need to prove that $N$ is an almost entropic matroid.

Let $K=\max_{u\in U}g(\fbr{u})$. We may assume that $K > 0$, since otherwise $N$ is a trivial matroid.

\begin{claim}\label{claim:exists-quasi-uniform}
    For every $\epsilon > 0$ and every $n_0\ge 2$ there exists a quasi-uniform random vector $\eta_U$ such that $n=\max_{u\in U}\lbr{\supp(\eta_u)}\ge n_0$, $n^{1/K}$ is an integer, and
    \[\forall A\seq U\quad \lbr{\frac{\Ent(\eta_A)}{\log_2(n)} - \frac{g(A)}{K}} \le \epsilon.\]
\end{claim}
\begin{proof}

Since $\clGent{U}$ is a cone, $\frac{1}{K}\cdot g$ is an almost entropic polymatroid. By approximation by quasi-uniform random variables~\cite{2008Yeung}, for every $\epsilon' > 0$ there exists $n\ge 2$, $c > 0$, and a quasi-uniform random vector $\eta_U$ such that $n=\max_{u\in U}\lbr{\supp(\eta_u)}$ and
\begin{equation}\label{eq:eta-almost-equal-g}
\forall A\seq U\quad \lbr{\frac{1}{c}\Ent(\eta_A) - \frac{g(A)}{K}} \le \epsilon'.
\end{equation}

Since $K=\max_{u\in U}g(\fbr{u})$, we have
\[\lbr{\frac{\log_2(n)}{c} - 1} \le \epsilon',\]
so
\[\forall A\seq U\quad \lbr{\frac{1}{c}\Ent(\eta_A) - \frac{1}{\log_2(n)}\Ent(\eta_A)}\le \frac{\Ent(\eta_A)}{\log_2(n)}\cdot\lbr{\frac{\log_2(n)}{c} - 1} \le \lbr{U}\cdot \epsilon'.\]

Thus, for $c'=\log_2(n)$, the inequality~\eqref{eq:eta-almost-equal-g} holds with $\wt{\epsilon}=(\lbr{U}+1)\epsilon'$.
Since $\epsilon'>0$ can be taken arbitrarily small, we may assume that $c=\log_2(n)$ in~\eqref{eq:eta-almost-equal-g}.

For every $r\ge 1$, we can take $r$ independent and identically distributed copies $(\eta_{U,i})_{i\in[r]}$ of $\eta_U$. Then the random vector $\eta'_U=((\eta_{u,i})_{i\in [r]})_{u\in U}$ is quasi-uniform, $\max_{u\in U}\lbr{\supp(\eta'_u)}=n^r$, and $\eta'_U$ satisfies \eqref{eq:eta-almost-equal-g} for $c'=r\cdot c$. Hence we may assume that $m=n^{1/K}$ is an integer and that $n$ can be taken arbitrarily large.
\end{proof}

Fix $\epsilon > 0$. Take a quasi-uniform random vector $\eta_U$ from Claim~\ref{claim:exists-quasi-uniform}.

We may assume that, for each $u\in U$, $\supp(\eta_u)\seq [n]$. Let $\mc{F}=\supp(\eta_U)\seq [n]^U$.

For each $u\in U$, take a uniformly random function $h_u:[n]\to [m]^{T_u}$. Let $h_U=(h_u)_{u\in U}:[n]^U\to [m]^V$.

With high probability, $h_U$ has few collisions.
\begin{claim}\label{claim:exists-good-h}
    For sufficiently large $n$, there exists $h_U$ such that, for every $C\seq V$ and $A\seq U$ satisfying $C\seq f(A)$ and $\rk_N(C)=g(A)$, we have
    \begin{equation}\label{eq:h-have-few-collisions}
        \forall X\in [m]^{C}\quad \lbr{\fbrc{Z\in \pi_A(\mc{F})}{\pi_C(h_A(Z))=X}} \le n^{4\epsilon}.
    \end{equation}
\end{claim}
\begin{proof}
    Fix $C\seq V$, $A\seq U$, and $X\in[m]^{C}$.

    For $Z\in \pi_A(\mc{F})$, define the random variable $\xi_Z=I(\pi_C(h_A(Z))=X)$.
    Let $\xi=\sum_{Z\in \pi_A(\mc{F})} \xi_Z$.

    We need to prove that $\P(\xi > n^{4\epsilon})$ is small. We do this by estimating the moments of $\xi$ by a method similar to that used in~\cite{2004Janson}.

    Take an arbitrary integer $r\in\Nz$ and arbitrary $Z_1,\ldots,Z_r\in\pi_A(\mc{F})$. We estimate the conditional expectation $\E(\xi\mid \forall i\in[r]\ \xi_{Z_i}=1)$.

    Take $Z\in \pi_A(\mc{F})$, and let $B=\fbrc{u\in A}{\exists i\in[r]\ \pi_u(Z)=\pi_u(Z_i)}\seq A$. Then
    \[\E(\xi_Z\mid \forall i\in[r]\ \xi_{Z_i}=1)=m^{-\lbr{C\setminus f(B)}}.\]
    We have
    \[g(A)=\rk_N(C)\le g(B) + \lbr{C\setminus f(B)},\]
    so
    \[\E(\xi_Z\mid \forall i\in[r]\ \xi_{Z_i}=1)\le m^{g(B)-g(A)}.\]

    Since $\eta_A$ is quasi-uniform, for each $B\seq A$ and $Y\in \pi_B(\mc{F})$, there are exactly $\frac{\lbr{\pi_A(\mc{F})}}{\lbr{\pi_B(\mc{F})}}$ choices of $Y'\in \pi_A(\mc{F})$ such that $\pi_B(Y')=Y$.

    Thus
    \[\E(\xi\mid \forall i\in[r]\ \xi_{Z_i}=1)\le \sum_{B\seq A} r^{\lbr{B}}\cdot \frac{\lbr{\pi_A(\mc{F})}}{\lbr{\pi_B(\mc{F})}}\cdot m^{g(B)-g(A)}\le \sum_{B\seq A} r^{\lbr{B}} \cdot n^{2\epsilon}=(r+1)^{\lbr{A}}\cdot n^{2\epsilon}.\]

    Then, for every $r\in\Nz$,
    \begin{align*}
        \E(\xi^{r+1})&=\sum_{Z_1,\ldots,Z_r\in \pi_A(\mc{F})}\P(\forall i\in[r]\ \xi_{Z_i}=1)\cdot \E(\xi\mid \forall i\in[r]\ \xi_{Z_i}=1)\\
        &\le (r+1)^{\lbr{A}}\cdot n^{2\epsilon} \cdot \sum_{Z_1,\ldots,Z_r\in \pi_A(\mc{F})}\P(\forall i\in[r]\ \xi_{Z_i}=1) = (r+1)^{\lbr{A}}\cdot n^{2\epsilon}\cdot \E(\xi^r).
    \end{align*}

    By induction, for every $q\in \Nz$,
    \[\E(\xi^q)\le (q!)^{\lbr{A}}\cdot n^{q\cdot 2\epsilon}\le (q^{\lbr{U}}\cdot n^{2\epsilon})^{q}.\]

    Take $a=n^{4\epsilon}$ and $q=\floor{n^{\epsilon/\lbr{U}}}$. By Markov's inequality,
    \[\P(\xi \ge a)\le (q^{\lbr{U}}\cdot n^{-2\epsilon})^q \le n^{-\epsilon\cdot q}=\exp(-\epsilon\ln(n)\cdot \floor{n^{\epsilon/\lbr{U}}}).\]
    This is superpolynomially small. The number of choices of $C$, $A$, and $X$ is at most
    $2^{\lbr{V}}\cdot 2^{\lbr{U}}\cdot n^{\lbr{V}/K}$, which is polynomial in $n$. Therefore, for sufficiently large $n$, there exists $h_U$ with the required properties.
\end{proof}

Take $h_U$ from Claim~\ref{claim:exists-good-h}. Let $\mc{G}=h_U(\mc{F})\seq[m]^V$.

By taking $C=V$ and $A=U$, we have
\[\forall G\in \mc{G}\quad \lbr{\fbrc{F\in\mc{F}}{h_U(F)=G}}\le n^{4\epsilon},\]
so
\[\lbr{\mc{G}}\ge \lbr{\mc{F}}\cdot n^{-4\epsilon}\ge n^{\frac{g(U)}{K}-5\epsilon}.\]

\begin{claim}\label{claim:G-prime-is-good}
    There exists $\mc{G}'\seq \mc{G}$ such that
    \[\lbr{\mc{G}'} \ge \lbr{\mc{G}}\cdot n^{-4\epsilon\cdot 2^{\lbr{V}}\cdot 2^{\lbr{U}}}\]
    and for each $C\seq V$ and $A\seq U$ such that $C\seq f(A)$ and $\rk_N(C)=g(A)$ we have
    \begin{equation}\label{eq:G-prime-is-good}
        \lbr{\pi_C(\mc{G}')}=\lbr{\pi_{f(A)}(\mc{G}')}.
    \end{equation}
\end{claim}
\begin{proof}
Take $C\seq V$ and $A\seq U$ such that $C\seq f(A)$ and $\rk_N(C)=g(A)$.

By \eqref{eq:h-have-few-collisions}, for each $X\in \pi_C(\mc{G})$ there are at most $n^{4\epsilon}$ choices of $Y\in \pi_{f(A)}(\mc{G})$ such that $\pi_C(Y)=X$. For each $X$, choose $Y_X\in \pi_{f(A)}(\mc{G})$ such that
\[\lbr{\mc{G}[Y_X]}\ge \lbr{\mc{G}[X]}\cdot n^{-4\epsilon},\]
where, for $D\seq V$ and $Z\in\pi_D(\mc{G})$, $\mc{G}[Z]=\fbrc{G\in\mc{G}}{\pi_D(G)=Z}$. We remove all $G\in \mc{G}$ such that $\pi_C(G)=X$ and $\pi_{f(A)}(G)\neq Y_X$. Doing this for each such $X$, we obtain $\mc{G}'$ with
\[\lbr{\mc{G}'}=\sum_{X\in \pi_C(\mc{G})}\lbr{\mc{G}[Y_X]}\ge \sum_{X\in \pi_C(\mc{G})}\lbr{\mc{G}[X]}\cdot n^{-4\epsilon}=\lbr{\mc{G}}\cdot n^{-4\epsilon}\]
and $\lbr{\pi_C(\mc{G}')}=\lbr{\pi_{f(A)}(\mc{G}')}$.

Repeating this procedure for each $C$ and $A$ gives the claim.
\end{proof}

\begin{claim}
    The family $\mc{G}'$ is an $\Flats(N)$-system.
\end{claim}
\begin{proof}
    Fix $C\seq D\seq V$ such that $\rk_N(C)=\rk_N(D)$. We need to prove that $\lbr{\pi_C(\mc{G}')}=\lbr{\pi_D(\mc{G}')}$.

    Take $A\seq U$ such that
    \[\rk_N(D)=g(A) + \lbr{D\setminus f(A)}.\]
    Then
    \[\rk_N(C)\le g(A) + \lbr{D\setminus f(A)} - \lbr{(D\setminus C)\setminus f(A)}.\]
    Therefore $D\setminus C\seq f(A)$, and $D\setminus f(A)=C\setminus f(A)$.

    Take $C'=C\cap f(A)$. Then
    \[\rk_N(C)\le \rk_N(C') + \lbr{C\setminus f(A)} \le g(A) + \lbr{C\setminus f(A)}=\rk_N(C),\]
    so $\rk_N(C')=g(A)$. By \eqref{eq:G-prime-is-good},
    \[\lbr{\pi_{C'}(\mc{G}')}=\lbr{\pi_{f(A)}(\mc{G}')}.\]
    Hence
    \[\lbr{\pi_{D}(\mc{G}')}\le \lbr{\pi_{f(A)\cup (C\setminus f(A))}(\mc{G}')}=\lbr{\pi_{C'\cup (C\setminus f(A))}(\mc{G}')}=\lbr{\pi_C(\mc{G}')}\le \lbr{\pi_{D}(\mc{G}')}.\]
\end{proof}

Thus we have an $\Flats(N)$-system $\mc{G}'\seq [m]^V$ such that $m\ge 2$ and
\[\lbr{\mc{G}'}\ge n^{\frac{g(U)}{K}-5\epsilon-4\epsilon\cdot 2^{\lbr{V}}\cdot 2^{\lbr{U}}}=m^{\rk(N)-\delta(\epsilon)}.\]
Here $\delta(\epsilon)\to 0$ as $\epsilon\to 0$.

Therefore $\alpha(\Flats(N))=\rk(N)$, and the proof follows from Theorem~\ref{thrm:alpha-le-eta} and Proposition~\ref{prop:matroid-rank-function-has-unique-solution}. For completeness, we give the remaining proof here.

Take a uniform random vector $\xi_V$ supported on $\mc{G}'$. Define a function $g':2^V\to \R$ by
\[\forall C\seq V\quad g'(C)=\frac{\Ent(\xi_C)}{\log_2(m)}.\]
Then $g'$ is an almost entropic $1$-polymatroid. Since $\mc{G}'$ is an $\Flats(N)$-system, $g'$ is compatible with $\Flats(N)$ and
\[g'(V) \ge \rk(N) - \delta(\epsilon).\]
Since $\delta(\epsilon)$ can be taken arbitrarily small and the family of almost entropic $1$-polymatroids is a closed bounded subset of $\R^{2^V}$, there exists an almost entropic $1$-polymatroid $\wt{g}:2^V\to \R$ compatible with $\Flats(N)$ such that $\wt{g}(V)=\rk(N)$. By Proposition~\ref{prop:matroid-rank-function-has-unique-solution}, we have $\wt{g}=\rk_N$, and therefore $N$ is an almost entropic matroid.

\section{Proof of Proposition~\ref{prop:counterexample-3-19-10}}\label{appendix:example-10}

\begin{proof}
    Let $U = \{1, 2, 3, 4, 5, 6, 7, 8, 9, 10\}$ and consider the following closure system $\mc{M}$ on $U$:
    \begin{gather*}
        \fbr{1, 2, 3, 4, 5, 6, 7, 8, 9, 10},\\
        \fbr{1, 4, 5, 6, 7, 9, 10},\\
        \fbr{1, 2, 4, 5, 6, 8},\\
        \fbr{1, 2, 7, 8, 9}, \fbr{1, 3, 4, 7, 9},\\
        \fbr{1, 3, 7, 8}, \fbr{1, 4, 5, 6}, \fbr{1, 4, 7, 9}, \fbr{1, 6, 8, 10},\\
        \fbr{1, 2, 8}, \fbr{1, 3, 7}, \fbr{1, 6, 8}, \fbr{1, 6, 10}, \fbr{1, 7, 8}, \fbr{1, 7, 9}, \fbr{2, 3, 10}, \fbr{2, 5, 6}, \fbr{3, 5, 10},\\
        \fbr{1, 4}, \fbr{1, 6}, \fbr{1, 7}, \fbr{1, 8}, \fbr{3, 6}, \fbr{3, 10}, \fbr{5, 6}, \fbr{5, 10},\\
        \fbr{1}, \fbr{2}, \fbr{3}, \fbr{5}, \fbr{6}, \fbr{10},\\
        \emptyset.
    \end{gather*}
    One can verify that this family is indeed intersection-closed and that $\rk(\mc{M}) = 3$.

    Let $(U,g)$ with $g:2^U\to\mathbb{R}$ be an arbitrary $1$-polymatroid compatible with $\mc{M}$. Multiplying each inequality in Figure~\ref{fig:ineq-system} by the coefficient shown in its first column, and then summing all the resulting inequalities, we obtain
    \[
        g(U) \le \frac{19}{10}.
    \]
    Hence $\rho(\mc{M}) \le \frac{19}{10}$, and by Corollary~\ref{corollary:alpha-le-rho}, $\alpha(\mc{M}) \le \frac{19}{10}$.

    \begin{figure}
        \centering
        \begin{alignat*}{4}
             & \frac{1}{10}\cdot\Big|\quad g(\fbr{2})                                                & \le & \ 1;                                                  \\
             & \frac{4}{10}\cdot\Big|\quad g(\fbr{3})                                                & \le & \ 1;                                                  \\
             & \frac{1}{10}\cdot\Big|\quad g(\fbr{5})                                                & \le & \ 1;                                                  \\
             & \frac{3}{10}\cdot\Big|\quad g(\fbr{10})                                               & \le & \ 1;                                                  \\
             & \frac{2}{10}\cdot\Big|\quad g(\fbr{1, 4})\ \Big(= g(\cl_{\mc{M}}(\fbr{4})) \Big)\  & \le & \ 1;                                                  \\
             & \frac{5}{10}\cdot\Big|\quad g(\fbr{1,8})\ \Big(= g(\cl_{\mc{M}}(\fbr{8}))\Big)\    & \le & \ 1;                                                  \\
             & \frac{3}{10}\cdot\Big|\quad g(\fbr{1,7,9})\ \Big(= g(\cl_{\mc{M}}(\fbr{9}))\Big)\  & \le & \ 1;                                                  \\
             & \frac{1}{10}\cdot\Big|\quad g(\fbr{1,2,7,8,9})                                        & \le & \ g(\fbr{2}) + g(\fbr{1,7});                    \\
             & \frac{4}{10}\cdot\Big|\quad g(\fbr{1,3,7})                                            & \le & \ g(\fbr{1}) + g(\fbr{3});                      \\
             & \frac{1}{10}\cdot\Big|\quad g(\fbr{1,4,5,6})                                          & \le & \ g(\fbr{5}) + g(\fbr{1});                      \\
             & \frac{3}{10}\cdot\Big|\quad g(\fbr{1,6,10})                                           & \le & \ g(\fbr{10}) + g(\fbr{1});                     \\
             & \frac{1}{10}\cdot\Big|\quad g(\fbr{1,4,5,6,7,9,10}) + g(\fbr{1,6})\                & \le & \ g(\fbr{1,4,5,6}) + g(\fbr{1,6,10});           \\
             & \frac{2}{10}\cdot\Big|\quad g(\fbr{1,2,4,5,6,8}) + g(\fbr{1})                      & \le & \ g(\fbr{1,4}) + g(\fbr{1,8});                  \\
             & \frac{3}{10}\cdot\Big|\quad g(\fbr{1,4,5,6,7,9,10}) + g(\fbr{1})                   & \le & \ g(\fbr{1,6}) + g(\fbr{1,7});                  \\
             & \frac{3}{10}\cdot\Big|\quad g(\fbr{1,7,8}) + g(\fbr{1})                            & \le & \ g(\fbr{1,8}) + g(\fbr{1,7});                  \\
             & \frac{4}{10}\cdot\Big|\quad g(\fbr{1,3,4,7,9}) + g(\fbr{1,7})                      & \le & \ g(\fbr{1,3,7}) + g(\fbr{1,7,9});              \\
             & \frac{1}{10}\cdot\Big|\quad g(U) + g(\fbr{1,7,9})                                  & \le & \ g(\fbr{1,2,7,8,9}) + g(\fbr{1,4,7,9});        \\
             & \frac{2}{10}\cdot\Big|\quad g(U) + g(\fbr{1,6})                                    & \le & \ g(\fbr{1,2,4,5,6,8}) + g(\fbr{1,6,10});       \\
             & \frac{3}{10}\cdot\Big|\quad g(U) + g(\fbr{1,7})                                    & \le & \ g(\fbr{1,4,7,9}) + g(\fbr{1,7,8});            \\
             & \frac{4}{10}\cdot\Big|\quad g(U) + g(\fbr{1,4,7,9})                                & \le & \ g(\fbr{1,3,4,7,9}) + g(\fbr{1,4,5,6,7,9,10}).
        \end{alignat*}
        \caption{Inequalities for $g$ with assigned coefficients}
        \label{fig:ineq-system}
    \end{figure}
\end{proof}

\end{document}